\documentclass[1 leqno,11pt,psfig]{amsart}
\usepackage{amssymb, amsmath,amsmath,latexsym,amssymb,amsfonts,amsbsy, amsthm,mathtools,graphicx,CJKutf8,CJKnumb,CJKulem}
\usepackage{amssymb}
\usepackage{color}
\usepackage{graphicx,epsfig}
\usepackage{amsmath}
\usepackage{mathrsfs}
\usepackage{amssymb}
\usepackage[colorlinks=true,citecolor=blue,bookmarks=true]{hyperref}

\usepackage{amssymb,mathrsfs,amsfonts,amsmath,amsbsy,tikz}
\usepackage{fontenc}
\usepackage{textcomp}

\numberwithin{equation}{section}

\allowdisplaybreaks

\newcommand{\beq}{\begin{equation}}
\newcommand{\eeq}{\end{equation}}
\newcommand{\ben}{\begin{eqnarray}}
\newcommand{\een}{\end{eqnarray}}
\newcommand{\beno}{\begin{eqnarray*}}
\newcommand{\eeno}{\end{eqnarray*}}

\newtheorem{theorem}{Theorem}[section]
\newtheorem{definition}[theorem]{Definition}
\newtheorem{lemma}[theorem]{Lemma}
\newtheorem{proposition}[theorem]{Proposition}
\newtheorem{corollary}{Corollary}[section]

\newtheorem{remark}[theorem]{Remark}

\begin{document}

\title{Expanding solutions near unstable Lane-Emden stars \footnote{To appear on Comm. Math. Phys.}}

\author{Ming Cheng}
\address{School of Mathematics, Jilin University, Changchun, 130012, China.}
\email{mcheng314@jlu.edu.cn}
\author{Xing Cheng}
\address{School of Mathematics, Hohai University, Nanjing 210098, Jiangsu, China}
\email{chengx@hhu.edu.cn}
\author{Zhiwu Lin}
\address{School of Mathematical Sciences, Fudan University, Shanghai, 200433, China}
\email{zwlin@fudan.edu.cn}

\date{}
\maketitle

\begin{abstract}

We consider the compressible Euler-Poisson equations for polytropes $P(\rho)=K\rho^{\gamma}$ with $\gamma\in \left(\frac{6}{5},\frac{4}{3} \right]$ and the
white dwarf stars. For $\gamma=\frac{4}{3},$ we establish the existence of a global weak solution for the spherically symmetric initial data with mass less than the  mass of the Lane-Emden stars (i.e. non-rotating polytropes). For $\gamma\in \left(\frac{6}{5},\frac{4}{3} \right)$, we show the existence of global weak
solution for spherical symmetric initial data in an invariant set containing a neighborhood of Lane-Emden stars. Moreover, the support of these solution
expands to infinity. As a corollary, this proves the strong instability of the Lane-Emden stars for $\gamma\in \left(  \frac{6}{5},\frac{4}{3}\right]  $.
For $\gamma\in \left(\frac{6}{5},\frac{4}{3} \right),$ our results provide the first example of expanding solutions near the Lane-Emden stars. For white dwarf stars,
we prove that the solution cannot collapse if the mass of initial data is less than the Chandrasekhar limit mass,
which is the supremum of the mass of the non-rotating white dwarf stars. Our proof strongly uses the variational characterization of the  Lane-Emden stars.
First, we relate the best constant of a Hardy-Littlewood type inequality with the mass of the Lane-Emden stars with $\gamma=\frac{4}{3}$,
which is further shown to equal the Chandrasekhar limit mass. For $\gamma\in\left(  \frac{6}{5},\frac{4}{3}\right)  $,
we show that the Lane-Emden stars are minimizers of an energy-mass functional subject to a Pohozaev type constraint.
This is crucial in the construction of the invariant set of expanding solutions.

\end{abstract}



\section{Introduction}

In this paper, we consider the  Newtonian self-gravitating gaseous stars modeled by the following compressible Euler-Poisson equations:
\begin{align}
&\partial_t \rho+\text{div} \left(\rho \textbf{u} \right)=0, \label{1.1}\\
&\partial_t  \left(\rho \textbf{u} \right)+\text{div} \left(\rho \textbf{u}\otimes \textbf{u} \right)+ \nabla P (\rho)
=-\rho \nabla V,\\
&\Delta V= 4\pi \rho,\label{1.3}\\
&\lim\limits_{\left\vert x\right\vert \rightarrow\infty}V(t,x)  = 0, \label{1.4}
\end{align}
where $(t,x)\in[0,+\infty)\times\mathbb{R}^3$, $\rho(t,x )\geq 0$ represents the density of the gas with the support  $\Omega(t)\subset \mathbb{R}^3$, $\textbf{u}( t, x )\in\mathbb{R}^3$ is the velocity vector, and $V(t, x )\in\mathbb{R}$ is the self-gravitational potential.
The enthalpy function $\Phi( \rho)$ is defined to be
\begin{align*}
\Phi(0) = \Phi'(0) = 0, \Phi''( \rho) = \frac{P'( \rho)}\rho.
\end{align*}
In particular, for $P(\rho) = K \rho^\gamma$ with $\gamma > 1 $, we have
$\Phi (\rho) = \frac1{\gamma  - 1} K \rho^\gamma.$

There are two important conserved quantities of \eqref{1.1}-\eqref{1.4}, which are the mass and the energy. The mass is defined to be
\begin{align*}
M(\rho ): = \int_{\mathbb{R}^3}\rho \, \mathrm{d}x,
\end{align*}
and the energy is
\begin{align*}
E(\rho, \mathbf{u} ) : =&\frac{1}{2}\int_{\mathbb{R}^3}\rho \left| \mathbf{u}   \right|^2 \, \mathrm{d}x
+\int_{\mathbb{R}^3} \Phi( \rho) \, \mathrm{d} x
-\frac{1}{2}\iint_{\mathbb{R}^3\times\mathbb{R}^3}\frac{\rho(x)\rho(y)}{|x-y|} \, \mathrm{d}x \mathrm{d}y.
\end{align*}
We first consider the polytropic case:
\begin{align}\label{eq1.5v34}
P(\rho)=K\rho^\gamma,
\end{align}
where $\gamma>\frac{6}{5}$ is the adiabatic exponent and $K$ is a positive constant.

The local well-posedness of the problem \eqref{1.1}-\eqref{eq1.5v34} with the physical vacuum condition  had been studied by X. Gu and Z. Lei \cite{GL},
T. Luo, Z. Xin and H. Zeng \cite{LXZ1}. We also refer to \cite{CS2012,JM} for related works on compressible Euler equations.
However, the existence of the global strong solutions to \eqref{1.1}-\eqref{eq1.5v34} is largely open. In the work of G.-Q. Chen, L. He, Y. Wang and D. Yuan \cite{GLD}, the global existence of the weak solutions to  \eqref{1.1}-\eqref{eq1.5v34} was proved for the spherically symmetric initial data with the mass less than a constant depending on the initial energy when $\gamma \in \big( \frac65, \frac43 \big]$ and with no restriction when $\gamma > \frac43$.
We also refer to the recent work of G.-Q. Chen, F. Huang, T. Li, W. Wang and Y. Wang \cite{CHW} on related results for more general equation of states including white dwarf stars.

Stellar dynamics of gaseous stars is a classical topic in astrophysics. Lane-Emden stars refer to non-rotating solutions of Euler-Poisson system with polytropic equation of states.
They are defined mathematically to be non-rotating spherically symmetric steady solutions  $(\rho_\mu(|x|),0)$  of \eqref{1.1}-\eqref{eq1.5v34} with the center density $\rho_\mu(0)=\mu$ $\in\left(  0,\infty\right)  $ and $0<R_\mu<\infty$ is the radius of $\text{supp} \, \rho_\mu$.
Therefore, $\rho_\mu$ satisfies
\begin{align}\label{1.11}
\begin{cases}
\nabla P(\rho_\mu)+\rho_\mu\nabla V_\mu=0,\ |x|<R_\mu,  \\
\Delta V_\mu=4\pi\rho_\mu.
\end{cases}
\end{align}
 The linear stability/instability of Lane-Emden stars had been studied mathematically by S.-S. Lin \cite{L}, who showed the linear instability for $\gamma \in  \left( \frac65, \frac43 \right)$ and linear stability for $\gamma \in \big[ \frac43, 2 \big)$.
The nonlinear stability/instability of non-rotating stars stars have been intensively studied in the last two decades.
J. Jang \cite{Jang08,Jang14} proved nonlinear instability of Lane-Emden stars for $ \gamma \in \left[ \frac65 ,  \frac43 \right)$. 

G. Rein \cite{R} proved the nonlinear conditional stability of Lane-Emden stars for $\gamma > \frac43$.
T. Luo and J. Smoller \cite{LS} proved the nonlinear conditional stability of non-rotating white dwarf stars when the mass is less than a critical mass.
In \cite{ZC2020}, Z. Lin and C. Zeng showed the turning point principle(TPP) for general equation of states, that is, the stability of the non-rotating stars is entirely determined by the mass-radius curve parameterized by the center density.
We also refer to \cite{CLW,HL,HLR,LW} for related works on TPP for other models.
In \cite{LWZ}, Z. Lin, Y. Wang and H. Zhu proved the nonlinear stability of the non-rotating stars from TPP for general equation of states. In particular, for spherically symmetric perturbations unconditional nonlinear stability is obtained in the sense that the existence of global weak solutions (the uniqueness is not known) is proved near the non-rotating stars.

There are some interesting works on the construction of collapsing or expanding solutions for Euler-Poisson systems. For the polytrope with $\gamma=\frac{4}{3}$, there exist self-similar collapsing or expanding solutions, see \cite{CS,GW,M2}.  It is worth noting that the energy of the self-similar collapsing solutions can be positive, negative, or zero (see \cite{CS}). M. Had\v{z}i\'c and J. Jang \cite{HJ18} proved the nonlinear stability of the expanding  self-similar solutions in \cite{GW}. We also refer to \cite{HJ190}, where the authors constructed a class of global expanding solutions for $\gamma \in (1, \frac{14}{13})$ or $\gamma = 1 + \frac1m$, $m\in \mathbb{N}\setminus \{1\}$. The constructed solutions have small initial data and compact support.
 For $\gamma \in \left(1, \frac43 \right)$,  Y. Guo, M. Had\v{z}i\'c and J. Jang \cite{GHJ} constructed a family of spherically symmetric collapsing solutions of \eqref{1.1}-\eqref{eq1.5v34} with initial mass close to delta mass, by using the perturbation of pressureless Euler-Poisson system.
In \cite{GHJ2}, Y. Guo, M. Had\v{z}i\'c and J. Jang strictly  proved the existence of the Larson-Penston type solution which is a self-similar solution describing the collapse of a self-gravitating asymptotically flat fluid with the isothermal case $\gamma=1$.
Recently, Y. Guo, M. Had\v{z}i\'c, J.  Jang and M. Schrecker \cite{GHJS} rigorously proved the existence of self-similar collapsing solutions but with infinite mass for the polytrope with  $\gamma \in \left( 1, \frac43\right)$. To the best of our knowledge, there is no construction of expanding solutions near the Lane-Emden stars for $\gamma \in \left( \frac65, \frac43\right)$. It should be pointed out that the viral type arguments used in \cite{YTTZ,MP} for constructing expanding solutions for $\gamma\geq\frac{4}{3}$ are not valid for $\gamma \in \left( \frac65, \frac43\right)$.

In the spherically symmetric setting, the support $\Omega(t)$ is a ball with radius $R(t)$. For $r=|x|\in (0,R(t))$, denote
\begin{align*}
\rho( t,x )=\rho(t, r ),\quad \textbf{u}(t, x )=u(t, r )\frac{x}{r},\quad V(t, x )=V(t, r ).
\end{align*}
Then, the system \eqref{1.1}-\eqref{1.4} can be rewritten as follows
\begin{align}\label{1.5}
&\partial_t\rho+\frac{1}{r^2}\partial_r \left(r^2\rho u \right)=0  &\text{in}\ (0,R(t)),\\
&\rho \left(\partial_tu+u\partial_ru \right)+\partial_r P+\frac{4\pi\rho}{r^2}\int^r_0\rho(t, s ) s^2 \,  \mathrm{d}  s=0 &\text{in}\ (0,R(t)),\label{1.6}
\end{align}
where $R(t)$ is the moving interface of fluids and vacuum states. We impose the kinematic condition:
\begin{align}
\frac{ \mathrm{d} }{ \mathrm{d} t}R(t)= u \left(t, R(t)  \right).
\end{align}
The density of gas $\rho(t, r )>0$ for $0\leq r < R(t)$ and
\begin{align}\label{1.8}
\rho(t, r )=0\ \text{for}\ r\geq R(t).
\end{align}
Our first main result is the following.
\begin{theorem}\label{th1.1}
For $\gamma = \frac43$, consider the problem \eqref{1.1}-\eqref{eq1.5v34}  with spherically symmetric initial data $(\rho_0,u_0)$ which has compact support and finite energy. Then, there exists a spherically symmetric global weak solution $(\rho,u)$ of  \eqref{1.1}-\eqref{eq1.5v34} when
 $M(\rho_0)< M_{K} $. Here, $M_{K}$ is the  mass of Lane-Emden stars for $\gamma=\frac{4}{3}$.
Moreover, if the spherically symmetric global strong solution of \eqref{1.5}-\eqref{1.8} exists with $M( \rho_0) < M_{K}$,
then the support radius
 $R(t)>0$ is bounded away from $0$ and
\begin{align}\label{1.14}
R^2 (t)
\geq \frac{E(\rho_0,u_0)}{M(\rho_0)}t^2+\frac{ 2\int_{\mathbb{R}^3} \rho_0 u_0 \cdot x \,\mathrm{d}x }
{M(\rho_0)}t+\frac{ \int_{\mathbb{R}^3} \rho_0 |x|^2 \,\mathrm{d}x }{M(\rho_0)}.
\end{align}
\end{theorem}
\begin{remark}
The threshold mass $M_{K}$ is sharp in the sense that for $\gamma = \frac43$ there exist self-similar collapsing solutions with mass greater than or equal to $M_{K}$. We refer to \cite{M2} for more details.
\end{remark}
\begin{remark}
The result in Theorem \ref{th1.1} also holds in the non-radial case when the mass of the initial data is less than the critical mass $M_{K}$, except for the existence of a non-radial global  weak solution. More precisely, the diameter of the support  $\Omega(t)$ expands to infinity as time goes to infinity for the global strong solution when the mass of the initial data is less than the critical mass.
\end{remark}
\begin{remark}
For $\gamma=\frac{4}{3}$,  in Theorem 1.5  of \cite{YTTZ}, when the mass is less than a ``critical'' mass $M_c$(see \eqref{eq3.9v63}),
it was proved that the star cannot collapse to a point. We can show that $M_{K}>M_c$. Moreover, when the mass is less than $M_{K}$, it is shown in Theorem \ref{th1.1} that the diameter of the support tends to infinity at least by the linear rate as time goes to infinity.
In \cite{CS}, C.-C. Fu and S.-S. Lin constructed linearly expanding self-similar solutions of \eqref{1.1}-\eqref{eq1.5v34} for $\gamma=\frac{4}{3}$ when the mass is less than $M_{K}$. This suggests that the linear expanding rate in Theorem 1.1 is optimal for general initial data with mass less than  $M_{K}$. In Proposition \ref{pro3.1}, we also show that the critical mass $M_K$ is greater than the ``critical'' mass used in \cite{GLD} to ensure the existence of global weak solutions. Moreover,  in Lemma \ref{le5.1v46} we prove that the critical mass $M_K$ with $K=2AB^{-\frac{4}{3}}$ (see \eqref{eq1.16v53}) equals the Chandrasekhar limit mass, which is the supremum of mass of non-rotating white dwarf stars.

\end{remark}
Next, we construct expanding solutions when $\gamma \in \left( \frac65, \frac43 \right)$. To state the result, we introduce the following notations.
Define
\begin{align*}
Q(\rho)&=3K\int_{\mathbb{R}^3}\rho^\gamma  \,  \mathrm{d} x
-\frac{1}{2}\iint_{\mathbb{R}^3\times\mathbb{R}^3}\frac{\rho(x)\rho(y)}{|x-y|} \,  \mathrm{d} x \mathrm{d} y.
\end{align*}
For any $\mu > 0$, define the action functional
\begin{align*}
S_\mu(\rho)
=\int_{\mathbb{R}^3}\frac{K}{\gamma-1}\rho^\gamma  \, \mathrm{d} x
-\frac{1}{2}\iint_{\mathbb{R}^3\times\mathbb{R}^3}\frac{\rho(x)\rho(y)}{|x-y|}  \, \mathrm{d} x \mathrm{d} y
-V_\mu(R_\mu)\int_{\mathbb{R}^3}\rho  \, \mathrm{d} x,
\end{align*}
where $V_\mu(R_\mu)$ is the value of $V_\mu(r)$ at $R_\mu$,
 and $R_\mu$ is the radius of the support of density $\rho_\mu$ of Lane-Emden star with center density $\rho_\mu(0) = \mu$.
 From Lemma 3.6 in \cite{ZC2020}, we have
\begin{align*}
& M_\mu=\int_{\mathbb{R}^3}\rho_\mu  \, \mathrm{d} x =M_1\mu^\frac{3\gamma-4}{2},\ R_\mu =R_1\mu^\frac{\gamma-2}{2},
\intertext{ and }
&  V_\mu(R_\mu) =-\frac{M_\mu}{R_\mu}=-\frac{M_1}{R_1} \mu^{\gamma-1},
\end{align*}
where $M_1=\int_{\mathbb{R}^3}{\rho_1 } \, \mathrm{d} x$ is the mass of Lane-Emden star with center density ${\rho}_1(0)=1$ and $R_1$ is the radius of the support set of $\rho_1$. By direct calculation, for $\lambda>0$, there holds $\frac{ \mathrm{d} S_{\mu} \left(\lambda^3\rho(\lambda x) \right)}{ \mathrm{d} \lambda}=\frac{1}{\lambda}Q \left(\lambda^3\rho(\lambda x) \right)$ and the density $\rho_\mu$ of the Lane-Emden star satisfies $Q(\rho_\mu)=0$ which is a Pohozaev type identity.

Our second main result is as follows.
\begin{theorem}\label{th1.2}
For $\gamma\in \left(\frac{6}{5},\frac{4}{3} \right)$,
consider the problem \eqref{1.1}-\eqref{eq1.5v34} with spherically symmetric initial data $(\rho_0,u_0)$ which has compact support and finite energy.
Assume that $(\rho_0,u_0)\in \mathcal{I}  $, where
\begin{align}\label{eq1.14v34}
\mathcal{I}
: =  \bigg\{( \rho, u ): Q(\rho)>0, M( \rho) <  \left( \frac{5 \gamma - 6}{ 2( \gamma - 1)} \right)^\frac{2( \gamma - 1)}{5 \gamma - 6} \left( \frac{4 - 3 \gamma }{5 \gamma - 6} \right)^\frac{4- 3\gamma}{5 \gamma - 6} l_1^{ \frac{2( \gamma - 1)}{5 \gamma - 6} } \frac{R_1}{M_1}   \left(E( \rho, u) \right)^\frac{3\gamma - 4}{5 \gamma- 6}\bigg\}
\end{align}
with $l_1 = S_1( \rho_1)$.
Then, there exists a spherically symmetric global weak solution $(\rho,u)$ of  \eqref{1.1}-\eqref{eq1.5v34}.
Moreover, if the  global strong solution of  \eqref{1.5}-\eqref{1.8} exists when  $(\rho_0,u_0)\in \mathcal{I}$, then the support radius $R(t)>0$ is bounded away from $0$ and
\begin{align}\label{1.15}
R^2(t)\geq  \frac{\Lambda }{M(\rho_0)}t^2 +\frac{ 2\int_{\mathbb{R}^3} \rho_0 u_0 \cdot x \,\mathrm{d}x}{M(\rho_0)}t
+\frac{ \int_{\mathbb{R}^3} \rho_0 |x|^2 \,\mathrm{d}x}{M(\rho_0)},
\end{align}
where $\Lambda$ is a positive constant.
\end{theorem}
\begin{remark}\label{re1.6}
From Lemma \ref{cor4.2} and \ref{lem4.1}, the set $\mathcal{I}$ is invariant under the flow of \eqref{1.1}-\eqref{eq1.5v34}.
Moreover, the set $\mathcal{I}$ is non-empty.
Indeed, for the non-rotating star $(\rho_\mu,0)$ and $0<\lambda<1$, there holds $ \left(\lambda^3\rho_\mu(\lambda x),0 \right)\in \mathcal{I}$ (see Remark \ref{re4.5}).
\end{remark}
\begin{remark}
The result in Theorem \ref{th1.2} also holds in the non-radial case when the initial data belongs to the set $\mathcal{I}$, except for  the existence of a non-radial global weak solution. More precisely,  the diameter of the support $\Omega(t)$ expands to infinity as time goes to infinity for the global strong solution when the initial data belongs to the set $\mathcal{I}$.
\end{remark}
\begin{remark}
The existence of the expanding solutions in Theorems \ref{th1.1} and \ref{th1.2} is conditional in the sense that we assume the global weak solutions have enough regularity to ensure the calculation related to the virial identities holds. It should be pointed out that the expansion results remain valid even in the presence of shock formation. While shock wave solutions are weak, we verify that they satisfy the Rankine-Hugoniot jump conditions, ensuring that the virial-type arguments used to establish expansion remain applicable. Specifically, in Remark \ref{re4.1}, we demonstrate that the virial identities hold for piecewise smooth solutions containing shocks, confirming that such discontinuities do not obstruct the expansion mechanism. This indicates that the regularity assumptions of the global solutions in Theorems \ref{th1.1} and \ref{th1.2} are quite mild, as the virial arguments apply beyond the smooth solutions. Moreover, the set of the expanding solutions for $M(\rho_0)<M_{K}$ with $\gamma= \frac{4}{3}$ or for the initial data $(\rho_0,u_0)\in \mathcal{I}$ with $\gamma\in \left(\frac{6}{5},\frac{4}{3} \right)$ is open, which shows a kind of nonlinear stability of the expanding solutions.
\end{remark}
We can also solve the Euler-Poisson equations for $t\in(-\infty,0]$. In fact, the equations \eqref{1.1}-\eqref{eq1.5v34} are invariant under the time reversal transformation: $t\rightarrow-t,\rho\rightarrow\rho,u\rightarrow-u$. From Theorem \ref{th1.1} and \ref{th1.2}, there exists a spherically symmetric global  weak solution $ \left(\tilde{\rho}(t,x),\tilde{u}(t,x) \right)$ of  \eqref{1.1}-\eqref{eq1.5v34} for $t \in[0, \infty)$ when the spherically symmetric initial data $(\rho_0,-u_0)$ with compact support and finite energy satisfies $M(\rho_0)< M_{K}$ for $\gamma= \frac{4}{3}$ or $(\rho_0,u_0)\in \mathcal{I}$ for $\gamma\in \left(\frac{6}{5},\frac{4}{3} \right)$. Then,
$(\rho(t,x),u(t,x))= \left(\tilde{\rho}(-t,x),- \tilde{u}(-t,x) \right)$ is a global weak solution of \eqref{1.1}-\eqref{eq1.5v34} for $t\in(-\infty,0]$ with the  initial data $(\rho_0,u_0)$.
Moreover, if  there exists a  global strong solution of  \eqref{1.5}-\eqref{1.8} for $t\in(-\infty,0]$ when the total mass $M(\rho_0)<M_{K}$ for $\gamma=\frac{4}{3}$ or $(\rho_0,u_0)\in \mathcal{I}$ for $\gamma\in \left(\frac{6}{5},\frac{4}{3} \right)$, then the virial arguments of Theorems \ref{th1.1} and \ref{th1.2} also hold for $t\leq0$. It obtains that the support radius of $(\rho,u)$ satisfies \eqref{1.14} or \eqref{1.15} for $t\in (-\infty,0]$ and consequently the support of the gaseous star expands to infinity as time goes to negative infinity. Therefore, we get
\begin{corollary}
The Lane-Emden stars $\left(  \rho_{\mu},0\right)  \ $for $\gamma\in\left(
\frac{6}{5},\frac{4}{3}\right]  $ are  strongly unstable in the sense that there exist initial data arbitrarily close to $\left(
\rho_{\mu},0\right)  $ in a smooth space such that the support of the solutions tend to infinity as time goes to both positive and negative
infinity provided the global strong solutions exist.
\end{corollary}
\begin{proof}
For any $\varepsilon>0$ small, choose the spherically symmetric initial data
\[
\left(  \rho_{0}^{\varepsilon},u_{0}^{\varepsilon}\right)  =\left\{
\begin{array}
[c]{cc}%
\left(  \left(  1-\varepsilon\right)  \rho_{\mu},0\right)   & \text{when
}\gamma=\frac{4}{3}\\
\left(  \left(  1-\varepsilon\right)  ^{3}\rho_{\mu}(\left(  1-\varepsilon
\right)  r),0\right)   & \text{when }\gamma\in\left(  \frac{6}{5},\frac{4}%
{3}\right)
\end{array}
\right.  .
\]
By Theorems \ref{th1.1} and \ref{th1.2}, when  a spherically symmetric global strong
solution  $\left(  \rho^{\varepsilon}\left(
t,r\right)  ,u^{\varepsilon}\left(  t,r\right)  \right)  $ with the initial data
$\left(  \rho_{0}^{\varepsilon},u_{0}^{\varepsilon}\right)  \ $ exists,  its support
expands to infinity as $t\rightarrow+\infty$ or $-\infty$. Moreover, by the virial argument in our proof, there holds
\begin{align*}
\int_{\mathbb{R}^3}|\rho^\epsilon(t,x)-\rho_\mu(x)||x|^2 \mathrm{d} x\rightarrow \infty,\ \text{as}\ t\rightarrow\pm\infty.
\end{align*}
\end{proof}
Below, we  outline  the proof of Theorems \ref{th1.1} and \ref{th1.2}.
\begin{itemize}
\item   When $\gamma = \frac43$, we consider  the variational problem:
\begin{align*}
\inf \left\{J( \rho): \rho \in L^\frac43 (\mathbb{R}^3)\cap L^1(\mathbb{R}^3),\rho\geq 0,\rho\not\equiv0\right\},
\end{align*}
where $J$ is the functional
\begin{align*}
J( \rho) = \frac{  \left(\int_{\mathbb{R}^3} \rho \,  \mathrm{d} x \right)^\frac23 \int_{\mathbb{R}^3} \rho^\frac43 \, \mathrm{d} x }{ \iint_{\mathbb{R}^3 \times \mathbb{R}^3} \frac{ \rho(x) \rho(y)}{|x- y|} \, \mathrm{d} x  \mathrm{d} y }.
\end{align*}
This variational problem is closely related to the best constant of the Hardy-Littlewood type inequality
\begin{align}\label{eq1.15v45}
\iint_{\mathbb{R}^3 \times \mathbb{R}^3} \frac{ \rho(x) \rho(y)}{ |x - y|}  \, \mathrm{d} x \mathrm{d} y
\leq C_{\min} \left(\int_{\mathbb{R}^3} \rho  \, \mathrm{d} x \right)^\frac23 \int_{\mathbb{R}^3} \rho^\frac43 \, \mathrm{d} x.
\end{align}
It is shown in Theorem \ref{thm3.1} that the minimum of $J$ and consequently the equality in \eqref{eq1.15v45} are attained by the Lane-Emden stars with $\gamma=\frac{4}{3}$ up to some scaling and translation. Moreover, we give the following relation between $C_{\min}$ and the  mass $M_{K}$ of the  Lane-Emden stars
\begin{equation}
C_{\min}=\frac{6K}{M_{K}^{\frac{2}{3}}}. \label{optimal-HL}%
\end{equation}
This is crucial in the proof of Theorem 1.1. In Proposition \ref{pro3.1}, we compare the different mass constants in results of  \cite{GLD,YTTZ,CS} with the critical mass $M_{K}$.
We find that the critical mass $M_0$ in \cite{CS} is equal to the critical mass $M_{K}$. On the other hand, the ``critical'' masses in \cite{GLD,YTTZ} are strictly less than $M_{K}$. By the optimal Hardy-Littlewood type inequality \eqref{eq1.15v45}, the existence of a global weak solution in the spherically symmetric case with mass less than $M_{K}$ is shown by using ideas in \cite{GLD}. More precisely, if the mass is less than the critical mass $M_{K}$, we can show the internal energy $\int_{\mathbb{R}^3}\rho^{\frac{4}{3}}dx$ is bounded. This leads to the key a priori estimates such as the basic energy estimate,
BD-type entropy estimate etc. of solutions of the corresponding Navier-Stokes-Poisson equations.
By similar arguments as in \cite{GLD}, it yields that the Navier-Stokes-Poisson equations have global strong solutions.
Then, we can get the global weak solution of the Euler-Poisson equations by taking the vanishing viscosity limit and using the compensated compactness arguments as in \cite{GLD}. The boundedness of the internal energy implies that the diameter of the support of the weak solution is bounded from below and the star cannot collapse. If additionally we assume that the global strong solution of  \eqref{1.5}-\eqref{1.8} exists, then the second time derivative of the virial $\int_{\mathbb{R}^3} \rho |x|^2 \, \mathrm{d}x$ is bounded below by the energy which is strictly positive when the mass is less than $M_{K}$. Therefore, the diameter of the support of the gaseous star tends to infinity as time goes to infinity.
\begin{remark}
After this work is completed, we learn that the best constant of the inequality \eqref{eq1.15v45} had also been studied by A. Blanchet, J. A. Carrillo, and P. Laurencot \cite{BCL} in their study of the Patlak-Keller-Segel system. They proved that the best constant in \eqref{eq1.15v45} is obtained by following the argument in \cite{Le}. But they did not relate the best constant to the mass of the Lane-Emden star, such as (\ref{optimal-HL}).
\end{remark}

\item When $ \gamma \in \left( \frac65, \frac43 \right) $, we consider the following constrained variational problem:
 \begin{align*}
 \inf \limits_{ \rho \in \mathcal{K}} S_\mu ( \rho),
 \end{align*}
where
\begin{align}\label{K}
\mathcal{K}&= \left\{ \rho  \in L^\gamma  (\mathbb{R}^3)\cap L^1(\mathbb{R}^3): Q(\rho)=0, \rho\geq 0,\rho \not\equiv 0\right\}.
\end{align}
The functional $S_\mu(\rho)$ under the constraint $\mathcal{K}$ is bounded below when $\gamma<\frac{4}{3}$.
In Theorem \ref{thm3.4}, we show that the minimum of this variational problem is attained by the Lane-Emden star $\left(  \rho_{\mu},0\right)  \ $ up to some translation. Then, we construct a family of non-empty invariant sets of the Euler-Poisson equations.
The union of the invariant sets is equivalent to $\mathcal{I}$ in \eqref{eq1.14v34}.
By similar arguments in \cite{GLD}, to show the existence of a global weak solution of Euler-Poisson system,
we first construct strong solutions of Navier-Stokes-Poisson equations. The set $\mathcal{I}$  is shown to be invariant under the flow generated by the Navier-Stokes-Poisson equations. By using this invariant property, we get the a priori estimates of the internal energy of the solutions of the Navier-Stokes-Poisson equations when the initial data belongs to $\mathcal{I}$. Then, by similar arguments for the case $\gamma=\frac{4}{3}$, we obtain the existence of the global weak solutions of the Euler-Poisson equations for initial data in $\mathcal{I}$.
Assume additionally that the  global strong solution of \eqref{1.5}-\eqref{1.8} exists when the initial data belongs to $\mathcal{I}$. Then the set $\mathcal{I}$ is invariant under the flow generated by the Euler-Poisson equations. Moreover, the second derivative of the virial $\int_{\mathbb{R}^3} \rho |x|^2 \, \mathrm{d}x$ is shown to be bounded below by the functional $Q$ and therefore remains strictly positive. Consequently, the diameter of the support of the solution expand to infinity as time tends to infinity.
\end{itemize}

\begin{remark}
The construction of expanding solutions near Lane-Emden stars is motivated by the variational approach used by  H. Berestycki and T. Cazenave \cite{BC}, which was originally developed to prove blow-up for the nonlinear Schr\"{o}dinger equations. See also the textbook \cite{CA} and the survey \cite{Ste} for detailed discussions of this method. But there are key differences between our approach and the classical argument in \cite{BC}. The variational arguments are used in \cite{BC} to construct an invariant set of initial data leading to blow-up in finite time for the nonlinear Schr\"{o}dinger equation. In our case, we constructed an invariant set of expanding solutions by using a similar variational characterization of the Lane-Emden stars.
However, such an argument cannot be applied to construct collapsing (blow-up) solutions for the Euler-Poisson system. The reason is that, unlike the case of the Schr\"{o}dinger equations, for the Euler-Poisson case there is an additional kinetic energy term in the time derivatives of the second inertia and this prevents the application of the similar arguments as in \cite{BC} to construct collapsing solutions.

 When $ \gamma =\frac43 $, expanding solutions are proved in \cite{YTTZ} for initial data with positive total energy. The proof implicitly assume the global existence of strong solutions for such initial data. However, the existence of self-similar collapsing solutions (see \cite{HJ18}) with positive energy shows that positive energy does not even guarantee the existence of global weak solutions. By contrast, for general initial data with mass less $M_{K}$ we proved the global existence of weak solutions and the expansion of the support by assuming slightly more regularity on the solutions.
%
\end{remark}

Now, we turn to the study of the white dwarf stars.   The pressure of the white dwarf star is
\begin{align}\label{eq1.16v53}
P_w(\rho)=Af(\xi ),\ \rho=B \xi^3,
\end{align}
where $A,B$ are two positive constants and
\begin{align*}
f(\xi )=  \xi \sqrt{\xi^2+1} \left(2 \xi^2-3 \right)+3\ln\Big( \xi+\sqrt{1+ \xi^2}\Big)
= 8\int^\xi_0\frac{u^4}{\sqrt{1+u^2}} \, \mathrm{d} u.
\end{align*}
We see that $P_w(\rho)\approx K \rho^\frac{4}{3}$ for large $\rho$ with $K=2AB^{-\frac{4}{3}}$ which shows the relation between the white dwarf stars and the polytropes with $\gamma=\frac{4}{3}$. The non-rotating white dwarf stars are known to have a maximum mass, which is called Chandrasekhar limit mass (\cite{Chan1939}).
We can show
\begin{theorem}\label{th5.2}
The Chandrasekhar limit mass $M_{ch}$ equals the  mass $M_K$ of the Lane-Emden stars when $P(\rho) = K \rho^\frac43$ with $K= 2AB^{- \frac43}$.
If the mass of the white dwarf star $M(\rho_0)$ is strictly less than $M_{ch}$, then the star cannot collapse to a point.
\end{theorem}
The fact that the Chandrasekhar limit mass  is exactly the  mass of the Lane-Emden star when $\gamma = \frac43$ is mentioned in \cite{Chan1939,Chan1984}. We provide a rigorous proof of this important fact in Lemma \ref{le5.1v46}.
By the asymptotic behavior of the equation of states of the white dwarf stars, we show that the internal energy of the white dwarf stars is controlled by the internal energy of the polytrope when $\gamma = \frac43$. Then by the optimal Hardy-Littlewood type inequality established in Theorem \ref{thm3.1}, we can show that the internal energy of the white dwarf stars is bounded when the mass is less than the Chandrasekhar limit mass $M_{ch}$. This ensures the white dwarf stars cannot collapse to a point. Recently in \cite{CHW}, G.-Q. Chen, F. Huang, T. Li, W. Wang and Y. Wang considered the gobal weak solutions of Euler-Poisson system \eqref{1.1}-\eqref{1.4} with  general equations of states including the white dwarf stars. In particular, for white dwarf stars and radially symmetric initial data with mass less than $M_{ch}$, by using our characterization of $M_{ch}$ and the relation (\ref{optimal-HL}) they obtained similar estimates for the internal energy of solutions of the Navier-Stokes-Poisson equations. Then the existence of global weak solution of Euler-Poisson system is proved by concentration-compactness arguments.

\begin{remark}
The white dwarf stars cannot expand like the polytrope with $\gamma=\frac{4}{3}$ when the mass is less than the Chandrasekhar limit mass. Indeed, there exists a family of non-rotating white dwarf stars with mass ranging from $0$ to $M_{ch}$.
\end{remark}
\textbf{Notation.}
We use $X \lesssim Y$ when $X \le CY$ for some constant $C> 0$, $X\gtrsim Y$ when $ C X \ge Y$, and $X \sim Y$ when $X \lesssim Y \lesssim X$. $B_R = B_R(0)\subset\mathbb{R}^3$ is a ball centered at the origin with radius $R> 0$.

The Fourier transform is defined to be
\begin{align*}
\hat{f} ( \xi) = \frac1{(2\pi)^\frac{3}2} \int_{\mathbb{R}^3} f(x) e^{-i x \cdot \xi} \,\mathrm{d}x.
\end{align*}
Let $L^p_{rad}(\mathbb{R}^3)= \left\{f\in L^p(\mathbb{R}^3):   f~ \text{is spherically symmetric} \right\}$.

This paper is organized as follows. In Section \ref{se2v28}, we prove some compactness lemmas which are crucial for the variational problems studied later.
In Section \ref{se3v28}, we consider the polytrope with $\gamma = \frac43$ and prove Theorem \ref{th1.1}.
In Section \ref{se4v28}, we consider the polytrope with $\gamma \in \left( \frac65, \frac43 \right)$ and prove Theorem \ref{th1.2}.
In Section \ref{se5v28}, we consider the white dwarf stars and prove Theorem \ref{th5.2}.
Lastly, the high dimensional cases {\color{red}are} studied in Section \ref{se6v28}.

\section{Preliminaries}\label{se2v28}
In this section, we prove some compactness results which are vital in the study of variational problems in Subsections \ref{subse3.1v31} and \ref{subse3.2v31}.

The solution of the Poisson equation
\begin{align*}
\begin{cases}
\Delta  V(x)=4\pi\rho(x),\ x\in\mathbb{R}^3,\\
\lim\limits_{|x|\rightarrow\infty}V(x)=0,
\end{cases}
\end{align*}
can be represented by
\begin{align*}
V(x)=-\int_{\mathbb{R}^3}\frac{\rho(y)}{|x-y|}\, \mathrm{d} y.
\end{align*}
Then, the gravitational potential energy  of the gaseous stars can be written as
\begin{align}\label{eq2.1v56}
\frac{1}{2}\iint_{\mathbb{R}^3\times\mathbb{R}^3}\frac{\rho(x)\rho(y)}{|x-y|}\, \mathrm{d} x \mathrm{d} y
= - \frac{1}{2}\int_{\mathbb{R}^3}\rho(x)  V(x)\,  \mathrm{d} x
&= \frac{1}{8\pi}\int_{\mathbb{R}^3} |\nabla V(x) |^2\,  \mathrm{d} x.
\end{align}
From the H\"older inequality and Sobolev embedding inequality, we have
\begin{align*}
\int_{\mathbb{R}^3}|\nabla V|^2\,  \mathrm{d} x
&=-4\pi\int_{\mathbb{R}^3}\rho V\,  \mathrm{d} x
\leq 4\pi \|\rho\|_{L^\frac{6}{5} \left(\mathbb{R}^3 \right)} \|V\|_{L^6 \left(\mathbb{R}^3 \right)}
\lesssim \|\rho\|_{L^\frac{6}{5} \left(\mathbb{R}^3 \right)} \|\nabla V\|_{L^2 \left(\mathbb{R}^3 \right)}.
\end{align*}
Therefore, for $\frac65 < \gamma \le \frac43  $, we have
\begin{align}\label{2.8}
\iint_{\mathbb{R}^3\times \mathbb{R}^3}\frac{\rho(x)\rho(y)}{|x-y|}\,  \mathrm{d} x \mathrm{d} y\lesssim
 \|\rho\|_{L^1 \left(\mathbb{R}^3 \right)}^{\frac{5\gamma-6}{3(\gamma-1)}}  \|\rho\|_{L^\gamma \left(\mathbb{R}^3 \right)}^{\frac{\gamma}{3(\gamma-1)}}.
\end{align}
The following  lemmas  play important roles in the study of variational problems later.
\begin{lemma}\label{lem2.1}
For $\gamma>1$, let $\{\rho_n\}\subset L^1 \left(\mathbb{R}^3 \right)\cap L^\gamma \left(\mathbb{R}^3 \right)$ be bounded and $\rho_n\geq 0$.
Then, there exists $\bar{\rho} \ge 0$  and a subsequence of $\{\rho_n\}$ which is still denoted by itself such that
\begin{align*}
\rho_n\rightharpoonup\bar{\rho}\ \text{in}\ L^\gamma \left(\mathbb{R}^3 \right).
\end{align*}
In addition, we have $\bar{\rho}\in L^1 \left(\mathbb{R}^3 \right)$ and
$$
 \left\|\bar{\rho} \right\|_{L^1 \left(\mathbb{R}^3 \right)}\leq \liminf_{n\to\infty} \|\rho_n\|_{L^1 \left(\mathbb{R}^3 \right)}.
$$
\end{lemma}
\begin{proof}
First, we claim $\bar{\rho}\geq 0$. Note that
\begin{align*}
\lim_{n\to\infty}\int_{\mathbb{R}^3}\rho_ng\,  \mathrm{d} x= \int_{\mathbb{R}^3}\bar{\rho}g \, \mathrm{d} x,\ \forall g\in L^{\gamma'} \left(\mathbb{R}^3 \right),
\end{align*}
where $\frac{1}{\gamma'}+\frac{1}{\gamma}=1.$
If the set $E=\bigcup\limits_{k<0} \left\{x\in\mathbb{R}^3:\bar{\rho}(x)<\frac{1}{k} \right\}$ satisfies $\textrm{meas} E>0$,
then there exists $k'<0$ such that $\text{meas}\left\{x\in\mathbb{R}^3:\bar{\rho}(x)<\frac{1}{k'} \right\}>0$.
There exists a closed set $E'\subset \left\{x\in\mathbb{R}^3:\bar{\rho}(x)<\frac{1}{k'} \right\}$ such that $0<\textrm{meas} E'< \infty$.
Define the characteristic function of a set $U$ by
\begin{align*}
\chi_{U}=\left\{
              \begin{array}{ll}
                1, & x\in U, \\
                0, & x\notin U.
              \end{array}
            \right.
\end{align*}
Taking $g=\chi_{E'}$, we have
\begin{align*}
0\leq \lim_{n\to\infty}\int_{\mathbb{R}^3}\rho_ng\,  \mathrm{d} x
=\int_{\mathbb{R}^3}\bar{\rho}g\,  \mathrm{d} x<0.
\end{align*}
It is a contradiction.

Now, we separate $\mathbb{R}^3$ to the sum of  spherical shells $\{T_k\}_{k\in \mathbb{Z}}$, where the measure of every spherical shell $T_k$ is $1$.
Let $\rho^{(k)}=\rho\chi_{T_k}$ and $\rho_n^{(k)}=\rho_n \chi_{T_k}$. Taking $g=\chi_{T_k}$, we have
\begin{align}\label{2.14}
\lim_{n\to\infty}\int_{\mathbb{R}^3}\rho^{(k)}_n \, \mathrm{d} x=\int_{\mathbb{R}^3}\bar{\rho}^{(k)} \, \mathrm{d} x.
\end{align}
Note that
\begin{align}\label{eq2.3v54}
\int_{\mathbb{R}^3}\bar{\rho}\,  \mathrm{d} x&=\sum_{k\in\mathbb{Z}}\int_{\mathbb{R}^3}\bar{\rho}^{(k)}\,  \mathrm{d} x
=\lim_{K\to\infty}\sum_{|k|\leq K}\int_{\mathbb{R}^3}\bar{\rho}^{(k)}\,  \mathrm{d} x.
\end{align}
For fixed $K$ large enough, by \eqref{2.14}, we have that for any $\epsilon>0$, there exists $N(\epsilon,K)>0$ such that
\begin{align}\label{eq2.4v54}
\sum_{|k|\leq K}\int_{\mathbb{R}^3}\bar{\rho}^{(k)}\,  \mathrm{d} x\leq \sum_{|k|\leq K}\int_{\mathbb{R}^3}\rho_n^{(k)}\,  \mathrm{d} x+\epsilon\leq \int_{\mathbb{R}^3}\rho_n\,  \mathrm{d} x+\epsilon,\ \forall n\geq N.
\end{align}
Because $\rho_n$ is bounded in $L^1 \left(\mathbb{R}^3 \right)$ for all $n\in\mathbb{N}$, by \eqref{eq2.3v54} and \eqref{eq2.4v54}, we have
\begin{align*}
\left\|\bar{\rho} \right\|_{L^1 \left(\mathbb{R}^3 \right)}\leq \liminf_{n\to\infty} \|\rho_n \|_{L^1 \left(\mathbb{R}^3 \right)},
\end{align*}
which completes the proof.
\end{proof}
\begin{lemma}\cite{R1}\label{lem2.2}
For  $\gamma>\frac{6}{5}$, we have for any $R'>R>0$, the mapping
\begin{align*}
\frac{1}{4\pi}\Delta  V(\rho)=\rho\in L^\gamma \left(\mathbb{R}^3 \right)\mapsto \chi_{B_{R'}}\nabla V \left(\chi_{B_R}\rho \right)\in L^2 \left(B_{R'} \right)
\end{align*}
is compact.
\end{lemma}
Now, we establish the compactness property of the spherically symmetric solution operator of the Poisson equation.
\begin{lemma}\label{lem2.3}
For $\gamma>\frac{6}{5}$, let  $\{\rho_n\}$ be a bounded sequence in $ L_{rad}^\gamma \left(\mathbb{R}^3 \right)\cap L_{rad}^1 \left(\mathbb{R}^3 \right)$, $\rho_n\geq 0$ and $\rho_n\rightharpoonup\bar{\rho}$ weakly in $L_{rad}^\gamma \left(\mathbb{R}^3 \right)$. Then, we have $\nabla V(\rho_n)\rightarrow \nabla V \left(\bar{\rho} \right)$ strongly in $L_{rad}^2 \left(\mathbb{R}^3 \right)$, where $\Delta  V= 4\pi\rho$.
\end{lemma}
\begin{proof}
Let $B^c_{R'}= \mathbb{R}^3\setminus B_{R'}$. For  $0<R<R'$, we have
\begin{align*}
\int_{B^c_{R'}} \left|\nabla V \left(\chi_{B_R}\rho_n \right) \right|^2 \, \mathrm{d}x
\lesssim\frac{1}{R'-R}   \|\chi_{B_R}\rho_n \|^2_{L^1 \left(\mathbb{R}^3 \right)}\lesssim\frac{1}{R'-R},\ \forall\, n\in\mathbb{N},
\end{align*}
which is sufficiently small for $R'>0$ large enough. Therefore, from Lemma \ref{lem2.2}, for any $R>0$, there holds
\begin{align}\label{2.20}
\nabla V \left(\chi_{B_R}\rho_n \right)\rightarrow \nabla V \left(\chi_{B_R}\bar{\rho} \right)\ \text{strongly in}\ L^2 \left(\mathbb{R}^3 \right).
\end{align}

By Lemma \ref{lem2.1} and the fact that $\{ \rho_n \}$ is bounded in $L^1$, we have
\begin{align*}
&\left\|\nabla V \left(\chi_{B^c_{R}}\rho_n \right)-\nabla V \left(\chi_{B^c_{R}}\bar{\rho} \right) \right\|_{L^2_{rad} (\mathbb{R}^3)}^2\\
=&4\pi\int^{ \infty}_0 \left|\partial_rV \left(\chi_{B^c_{R}}\rho_n \right)-\partial_rV  \left(\chi_{B^c_{R}}\bar{\rho} \right) \right|^2 r^2  \, \mathrm{d} r\\
=&(4\pi)^3\int^{ \infty}_0\Big|\frac{1}{r^2}\int^r_0\chi_{B^c_R}(\rho_n-\bar{\rho})s^2 \, \mathrm{d} s\Big|^2 r^2  \, \mathrm{d} r\\
\lesssim &\int^{\infty}_R\Big|\frac{1}{r^2}\int^r_0\chi_{B^c_R} (\rho_n-\bar{\rho}) s^2 \, \mathrm{d} s\Big|^2r^2 \, \mathrm{d} r
+\int^R_0\Big|\frac{1}{r^2}\int^r_0\chi_{B^c_R}(\rho_n-\bar{\rho})s^2  \, \mathrm{d} s\Big|^2r^2  \, \mathrm{d} r\\
\lesssim&\int^{ \infty}_R\Big|\frac{1}{r^2}\int^r_R(\rho_n-\bar{\rho})s^2  \, \mathrm{d} s\Big|^2 r^2  \, \mathrm{d} r
\lesssim \int^{ \infty}_R\frac{1}{r^2} \, \mathrm{d} r \lesssim \frac{1}{R}.
\end{align*}
Thus, for any given $ \epsilon>0$, there exists $R>0$ large enough such that
\begin{align}\label{2.30}
\left\|\nabla V \left(\chi_{B^c_{R}}\rho_n \right)-\nabla V \left(\chi_{B^c_{R}}\bar{\rho} \right) \right\|_{L^2_{rad}  \left(\mathbb{R}^3 \right)}<\epsilon,\ \forall\,  n.
\end{align}
By the triangle inequality, \eqref{2.20} and \eqref{2.30},
we get $\nabla V(\rho_n)\rightarrow \nabla V \left(\bar{\rho} \right)$ strongly in $L^2_{rad} \left(\mathbb{R}^3 \right)$ as $n \to \infty$.
\end{proof}
\begin{remark}
The results in this section can be generalized to higher dimensions by similar arguments.
\end{remark}

\section{ Expanding solutions for polytropes with $\gamma = \frac43$ }\label{se3v28}
In this section, we study the  polytrope for $\gamma=\frac{4}{3}$.

\subsection{The best constant of Hardy-Littlewood inequality}\label{subse3.1v31}
In this subsection, we will seek the best constant $C_{\min}$ satisfying the Hardy-Littlewood
type inequality
\begin{align}\label{2.9}
\iint_{\mathbb{R}^3\times \mathbb{R}^3}\frac{\rho(x)\rho(y)}{|x-y|}  \, \mathrm{d} x \mathrm{d} y\leq C_{\min}||\rho||_{L^1(\mathbb{R}^3)}^{\frac{2}{3}}||\rho||_{L^\frac{4}{3} \left(\mathbb{R}^3 \right)}^{\frac{4}{3}}.
\end{align}
We prove that the best constant of \eqref{2.9} is attained by the density of the Lane-Emden stars for $\gamma=\frac{4}{3}$ and it can be expressed explicitly by the  mass $M_{K}$ of the Lane-Emden stars.

Consider the closely related variational problem
\begin{align}\label{3.2}
J=\inf \left\{J(\rho): \rho\in L^\frac{4}{3} \left(\mathbb{R}^3 \right)\cap L^1 \left(\mathbb{R}^3 \right), \rho \ge 0, \rho \not\equiv 0\right\},
\end{align}
where
\begin{align*}
J(\rho)
=\frac{ \left(\int_{\mathbb{R}^3}\rho\,  \mathrm{d} x \right)^\frac{2}{3}\int_{\mathbb{R}^3}
\rho^\frac{4}{3} \, \mathrm{d} x}{\iint_{\mathbb{R}^3\times\mathbb{R}^3}\frac{\rho(x)\rho(y)}{|x-y|} \, \mathrm{d} x \mathrm{d} y}.
\end{align*}
\begin{theorem}\label{thm3.1}
There exists a non-negative, spherically symmetric and non-increasing function $\bar{\rho}\in L^\frac{4}{3} \left(\mathbb{R}^3 \right)\cap L^1 \left(\mathbb{R}^3 \right)$ such that $J=J \left(\bar{\rho} \right)$ and
\begin{align*}
 C_{\min}= J^{-1} = \frac{6K}{ M_{K}^\frac{2}{3}},
\end{align*}
where $M_{K}$ is the mass of the Lane-Emden stars for $\gamma=\frac{4}{3}$.
Moreover, if $\rho(x)$ is a minimizer of the variational problem \eqref{3.2}, there are some $R_* > 0$ and $x_* \in \mathbb{R}^3$ such that
\begin{align*}
\rho(x) =
\begin{cases}
\frac1{R_*^3 }    \rho_1 \left( \frac{x- x_*}{ R_*}  \right),  & \text{ if } x \in B_{R_*} \left(x_* \right), \\
0,                                                                     &  \text{ if } x \in  \left( B_{R_*} \left(x_* \right) \right)^c.
\end{cases}
\end{align*}
Here, $ \rho_1 $ is the density of the classical Lane-Emden star for $\gamma=\frac{4}{3}$ with the support radius $R_1<\infty$ and the center density $\rho_1(0)=1$.
\end{theorem}

\begin{proof}
We can see that $J$ is invariant under the following scaling:
$\varrho(x)= \lambda_1^\beta\rho \left(\lambda_2^\alpha x \right)$, $\lambda_1,\lambda_2>0$.
Indeed, by
\begin{align*}
&\Big(\int_{\mathbb{R}^3}\varrho(x) \,  \mathrm{d} x\Big)^\frac{2}{3}
=\Big(\int_{\mathbb{R}^3}\lambda_1^\beta \rho \left( \lambda_2^\alpha x \right)  \, \mathrm{d} x\Big)^\frac{2}{3}
=\lambda_1^{\frac{2}{3}\beta}\lambda_2^{-2\alpha}\Big(\int_{\mathbb{R}^3}\rho(x)  \, \mathrm{d} x\Big)^\frac{2}{3},\\
&\int_{\mathbb{R}^3}\varrho^\frac{4}{3} \, \mathrm{d} x
=\lambda_1^{\frac{4}{3}\beta}\lambda_2^{-3\alpha}\int_{\mathbb{R}^3}\rho^\frac{4}{3} \, \mathrm{d} x,\\
&\iint_{\mathbb{R}^3\times\mathbb{R}^3}\frac{\varrho(x)\varrho(y)}{|x-y|} \, \mathrm{d} x \mathrm{d} y
=\lambda_1^{2\beta}\lambda_2^{-5\alpha}\iint_{\mathbb{R}^3\times\mathbb{R}^3}\frac{\rho(x)\rho(y)}{|x-y|} \, \mathrm{d} x \mathrm{d} y,
\end{align*}
we have
\begin{align*}
J = \frac{ \left(\int_{\mathbb{R}^3}\rho  \, \mathrm{d} x \right)^\frac{2}{3}
\int_{\mathbb{R}^3} \rho^\frac{4}{3} \, \mathrm{d} x}{\iint_{\mathbb{R}^3\times\mathbb{R}^3}\frac{\rho(x)\rho(y)}{|x-y|}\,  \mathrm{d} x \mathrm{d} y}
=\frac{ \left(\int_{\mathbb{R}^3}\varrho  \, \mathrm{d} x \right)^\frac{2}{3}
\int_{\mathbb{R}^3}\varrho^\frac{4}{3} \, \mathrm{d} x}{\iint_{\mathbb{R}^3\times\mathbb{R}^3}\frac{\varrho(x)\varrho(y)}{|x-y|} \, \mathrm{d} x \mathrm{d} y}.
\end{align*}
It is obvious that $J>0$ from \eqref{2.9}.
From \cite{EM}, the symmetric-decreasing rearrangement $\rho^\sharp$ of $\rho$ satisfies
 \begin{itemize}
\item  $\rho^\sharp$ is non-negative.\\
 \item $\rho^\sharp$ is spherically symmetric and non-increasing, i.e. $\rho^\sharp(x)=\rho^\sharp(y)$ if $|x|=|y|$ and $\rho^\sharp(x)\geq \rho^\sharp(y)$ if $|x|\leq |y|$.\\
\item  $||\rho||_{L^p \left(\mathbb{R}^3 \right)}= \left\|\rho^\sharp \right\|_{L^p  \left(\mathbb{R}^3 \right)}$ for any $\rho\in L^p \left(\mathbb{R}^3 \right)$, where $1\leq p\leq \infty.$\\
\item \begin{align}
\iint_{\mathbb{R}^{3}\times\mathbb{R}^{3}}\frac{\rho(x)\rho(y)}{|x-y|}%
dxdy\leq\iint_{\mathbb{R}^{3}\times\mathbb{R}^{3}}\frac{\rho^{\sharp}%
(x)\rho^{\sharp}(y)}{|x-y|}dxdy.\label{rearrangement-Gra-energy}
\end{align}
\end{itemize}
Hence,
\begin{align*}
J& =\inf \left\{J(\rho): \rho\in L^\frac{4}{3} \left(\mathbb{R}^3 \right)\cap L^1 \left(\mathbb{R}^3 \right), \rho \ge 0, \rho \not\equiv 0\right\}\\
 & =\inf \left\{J(\rho): \rho\in L_{rad}^\frac{4}{3} \left(\mathbb{R}^3 \right)\cap L_{rad}^1 \left(\mathbb{R}^3 \right), \rho\geq 0,\rho\not\equiv 0\right\}.
\end{align*}
Let $\{\rho_n\}$ be the spherically symmetric minimizing sequence with $\rho \ge 0, \rho \not\equiv 0$ such that
\begin{align*}
J(\rho_n)\rightarrow J,\ \text{as}\ n\rightarrow\infty.
\end{align*}
Using the homogeneity and scaling invariance, we can assume that
$$||\rho_n||_{L^\frac{4}{3}_{rad}  \left(\mathbb{R}^3 \right )}=1
\text{ and }  ||\rho_n||_{L^1_{rad}  \left(\mathbb{R}^3 \right)}=1.$$
From Lemmas \ref{lem2.1} and \ref{lem2.3}, there exist $\bar{\rho}\geq 0~ (\not\equiv 0)$ and a subsequence of $\{\rho_n\}$ which is still denoted by itself such that
\begin{align*}
&\rho_n\rightharpoonup\bar{\rho} \ \text{in}\ L^\frac{4}{3}_{rad} \left(\mathbb{R}^3 \right), \quad
\left\|\bar{\rho} \right\|_{L^1_{rad} (\mathbb{R}^3)}\leq \liminf_{n\to\infty}  ||\rho_n||_{L^1_{rad} (\mathbb{R}^3)},
\intertext{ and }
&\nabla V(\rho_n)\rightarrow \nabla V \left(\bar{\rho} \right)\ \text{in}\ L^2_{rad}(\mathbb{R}^3).
\end{align*}
Therefore,
\begin{align*}
J(\bar{\rho})\leq \liminf_{n\to\infty}J(\rho_n)=J,
\end{align*}
which implies the minimum $J$ of \eqref{3.2} is obtained and obviously $(C_{\min})^{-1}=J $. We note that any minimizer of the variational problem \eqref{3.2} must be spherically symmetric with respect to some point in $\mathbb{R}^3$, and decreasing as a function of the radial variable. Indeed, let $\bar{\rho}$ be a minimizer of \eqref{3.2}, then by
(\ref{rearrangement-Gra-energy}) $J\left(  \bar{\rho}\right)  \geq J\left(
\bar{\rho}^{\sharp}\right)  $. Thus $J\left(  \bar{\rho}\right)  =J\left(
\bar{\rho}^{\sharp}\right)  $ and the equality in (\ref{rearrangement-Gra-energy})
holds true for $\bar{\rho}$. By Theorem 3.9 of \cite{EM}, $\bar{\rho}$ must be
spherically symmetric up to some translation.
%

Now, we derive the Euler-Lagrange equation for the minimizers of \eqref{3.2}. It suffices to consider a spherically symmetric minimizer $\bar{\rho}$. For $\epsilon>0$ small enough, define
\begin{align*}
\mathcal{S}_\epsilon:=\{x\in\mathbb{R}^3: \epsilon<\bar{\rho}<\frac{1}{\epsilon}\}.
\end{align*}
Take a test function $w\in L^\infty(\mathbb{R}^3)$ which has compact support and is non-negative on $\mathbb{R}^3\setminus \mathcal{S}_\epsilon$. For $\tau\geq 0$, define
\begin{align*}
\rho_\tau=\bar{\rho}+\tau w.
\end{align*}
Then, $\rho_\tau\geq 0$ for $\tau$ small enough. Since $\bar{\rho}$ is a minimizer of $J(\rho)$, we have
\begin{align*}
0\leq& J(\rho_\tau)-J(\bar{\rho})\\
=&\tau\bigg[\iint_{\mathbb{R}^3\times\mathbb{R}^3}\frac{\bar{\rho}(x)\bar{\rho}(y)}{|x-y|}  \, \mathrm{d} x \mathrm{d} y
\bigg(\frac{2}{3}\int_{\mathbb{R}^3}w \, \mathrm{d} x \Big(\int_{\mathbb{R}^3} \bar{\rho} \, \mathrm{d} x \Big)^{- \frac{1}{3} } \int_{\mathbb{R}^3}\bar{\rho}^\frac{4}{3}\,  \mathrm{d} x+\int_{\mathbb{R}^3} \frac{4}{3} \bar{\rho}^{ \frac{1}{3}} w \, \mathrm{d} x\Big(\int_{\mathbb{R}^3} \bar{\rho} \, \mathrm{d} x \Big)^{\frac{2}{3} }\bigg)\\
&-2 \left(\int_{\mathbb{R}^3}\bar{\rho} \, \mathrm{d} x \right)^\frac{2}{3}\int_{\mathbb{R}^3}\bar{\rho}^\frac{4}{3} \, \mathrm{d} x \iint_{\mathbb{R}^3\times\mathbb{R}^3}
\frac{w(x)\bar{\rho}(y)}{|x-y|}\,  \mathrm{d} x \mathrm{d} y\bigg]
\Big[\iint_{\mathbb{R}^3\times\mathbb{R}^3}\frac{\bar{\rho}(x)\bar{\rho}(y)}{|x-y|} \, \mathrm{d} x \mathrm{d} y\Big]^{-2}+o(\tau).
\end{align*}
Therefore, the  coefficient of $\tau$  must be non-negative. That is
\begin{align*}
\int_{\mathbb{R}^3}\left(A_1+A_2\bar{\rho}^\frac{1}{3}+A_3V(\bar{\rho})\right) w\, \mathrm{d} x\geq 0,
\end{align*}
where
\begin{align*}
&A_1=\frac{2}{3}\iint_{\mathbb{R}^3\times\mathbb{R}^3} \frac{\bar{\rho}(x)\bar{\rho}(y)}{|x-y|} \, \mathrm{d} x \mathrm{d} y
\Big(\int_{\mathbb{R}^3}\bar{\rho} \, \mathrm{d} x\Big)^{- \frac{1}{3}}
\int_{\mathbb{R}^3}\bar{\rho}^\frac{4}{3}  \, \mathrm{d} x,\\
&A_2=\frac{4}{3}\iint_{\mathbb{R}^3\times\mathbb{R}^3}\frac{\bar{\rho}(x)\bar{\rho}(y)}{|x-y|} \, \mathrm{d} x \mathrm{d} y\Big(\int_{\mathbb{R}^3}\bar{\rho} \, \mathrm{d} x\Big)^{\frac{2}{3}},\quad
A_3=2 \left(\int_{\mathbb{R}^3}\bar{\rho} \,  \mathrm{d} x \right)^\frac{2}{3}\int_{\mathbb{R}^3}\bar{\rho}^\frac{4}{3} \, \mathrm{d} x.
\end{align*}
Hence, $A_1+A_2\bar{\rho}^\frac{1}{3}+A_3V(\bar{\rho})=0$ in $\mathcal{S}_\epsilon$ and $A_1+A_2\bar{\rho}^\frac{1}{3}+A_3V(\bar{\rho})\geq 0$ in $\mathbb{R}^3\setminus\mathcal{S}_\epsilon$ for all $\epsilon>0$ small enough. Taking $\epsilon\rightarrow 0$, we get
\begin{align}\label{3.15}
A_1+A_2\bar{\rho}^\frac{1}{3}+A_3V(\bar{\rho})=0,\ \text{in}\ \Sigma_+,
\end{align}
where $\Sigma_+=\{x\in\mathbb{R}^3:\bar{\rho}(x) >0\}$.

After scaling $\tilde{\rho}(x)=\left( \frac{4KA_3}{A_2} \right)^\frac{3}{2}\bar{\rho}(x)$, \eqref{3.15} becomes
\begin{align}\label{tilde}
4K\tilde{\rho}^\frac{1}{3}+V(\tilde{\rho})=B_1,\ \text{in}\ \tilde{\Sigma}_+,
\end{align}
where $\tilde{\Sigma}_+= \{x\in\mathbb{R}^3:\tilde{\rho}(x) >0\}$ and $-B_1 = \frac{A_1}{A_3}\big(\frac{4KA_3}{A_2}\big)^\frac{3}{2}$.

 Since the terms $\tilde{\rho}^\frac{1}{3}$ and $V \left(\tilde{\rho} \right)$ in \eqref{tilde} belong to $L^4_{rad} \left(\mathbb{R}^3 \right)$, it follows that the support of $\tilde{\rho}$ is compact. Note that $\tilde{\rho}$ is spherically symmetric and non-increasing. There exists $0<R<\infty$ such that $\tilde{\rho}(|x|)=0 \left(x\in  B_R^c \right)$. From the regularity argument in \cite{GT}, we have $\tilde{\rho}^\frac{1}{3}\in C^2(B_R )$. Let $F(s)= \left(\Phi' \right)^{-1}(s) = \left( \frac{s}{4 K } \right)^3 $ for $s\in(0, \infty)$, where $\Phi$ is the enthalpy function. We extend $F(s)$ to $(-\infty,0)$ by zero extension. The extended function is denoted by $F_+(s):\mathbb{R}\rightarrow [0, \infty)$. Therefore, the potential $V \left(\tilde{\rho} \right)(r)=:V(r)$ satisfies
\begin{align*}
\Delta  V=V''+\frac{2}{r}V'=4\pi F_+(B_1-V).
\end{align*}
Let $y(r)=B_1-V(r)= 4K\tilde{\rho}^\frac{1}{3}=\Phi' \left(\tilde{\rho} \right).$
Thus, we obtain the following ODE:
\begin{align}\label{ODE}
\begin{cases}
y''+\frac{2}{r}y'=-4\pi F_+(y),\\
y(0)=\Phi' \left(\tilde{\rho}(0) \right) ,   \ y'(0)=0.
\end{cases}
\end{align}
This system is equivalent to
\begin{align*}
y'(r)=-\frac{4\pi}{r^2}\int^r_0s^2F_+(y(s)) \, \mathrm{d} s,\ y(0)=4K \left(\tilde{\rho}(0) \right)^\frac{1}{3}.
\end{align*}
From the classical ODE theory, this problem admits a unique solution $y(r)$.
We have $y(R)=4K\tilde{\rho}\left(  R\right)  ^{\frac{1}{3}}=0$ which implies $B_1=V(R)=-\frac{M_1}{R_1} \mu^\frac{1}{3}$, where $\mu=\tilde{\rho}(0)$. Then, $\tilde{\rho}$ is the solution of \eqref{1.11} for $\gamma=\frac{4}{3}$ with compact support $B_R$. Therefore, $\rho_1(|x|)=\mu^{-1}\tilde{\rho}\big(\mu^{-\frac{1}{3}}|x|\big)$ is the solution of \eqref{1.11} for $\gamma=\frac{4}{3}$ with the compact support $\left\{|x|\leq R_1\right\}$ and the center density $\rho_1(0)=1$.

In all, if $\rho$ is a minimizer of the variational problem $\eqref{3.2}$,
there must exist $x_* \in \mathbb{R}^3$ and $R_*>0$ such that $ {\rho}(x) =  \frac{1}{R_*^3}\rho_1  \left(\frac{x - x_*}{R_*} \right)$.

We now relate the minimum of the variational problem \eqref{3.2} to the mass of the Lane-Emden stars.
Multiplying the first equation in \eqref{1.11} with $x$ and integrating, we have
\begin{align*}
\int_{\mathbb{R}^3} x\cdot\nabla  \left(K\rho_\mu^\frac{4}{3} \right) \, \mathrm{d} x
+\int_{\mathbb{R}^3}\rho_\mu x\cdot\nabla V_\mu \,  \mathrm{d} x=0.
\end{align*}
There holds
\begin{align}\label{3.31}
3\int_{\mathbb{R}^3}  K\rho_\mu^\frac{4}{3} \, \mathrm{d} x  =\int_{\mathbb{R}^3}\rho_\mu x\cdot\nabla V_\mu  \, \mathrm{d} x.
\end{align}
On the other hand, note that
\begin{align*}
V_\mu(x)=-\int_{\mathbb{R}^3}\frac{\rho_\mu(y)}{|x-y|} \, \mathrm{d} y,\quad \partial_{x_i} V_\mu(x)
 =\int_{\mathbb{R}^3}\frac{\rho_\mu(y)(x_i-y_i)}{|x-y|} \, \mathrm{d} y.
\end{align*}
Then, by direct calculation, we have
\begin{align*}
\int_{\mathbb{R}^3}\rho_\mu x\cdot\nabla V_\mu \,  \mathrm{d} x
=&\sum_{i=1}^3\iint_{\mathbb{R}^3\times\mathbb{R}^3}\frac{\rho_\mu(x)\rho_\mu(y)(x_i-y_i)x_i}{|x-y|}  \, \mathrm{d} x \mathrm{d} y\\
=&\iint_{\mathbb{R}^3\times\mathbb{R}^3}\frac{\rho_\mu(x)\rho_\mu(y)}{|x-y|} \, \mathrm{d} x \mathrm{d} y
-\sum_{i=1}^3\iint_{\mathbb{R}^3\times\mathbb{R}^3}\frac{\rho_\mu(x)\rho_\mu(y)(x_i-y_i)y_i}{|x-y|}\,  \mathrm{d} x \mathrm{d} y\\
=&\iint_{\mathbb{R}^3\times\mathbb{R}^3}\frac{\rho_\mu(x)\rho_\mu(y)}{|x-y|} \, \mathrm{d} x \mathrm{d} y
-\int_{\mathbb{R}^3}\rho_\mu x\cdot\nabla V_\mu  \, \mathrm{d} x.
\end{align*}
Combing with \eqref{3.31}, we obtain
\begin{align}\label{Q}
3K\int_{\mathbb{R}^3}\rho_\mu^\frac{4}{3} \, \mathrm{d} x
=\frac{1}{2}\iint_{\mathbb{R}^3\times\mathbb{R}^3}\frac{\rho_\mu(x)\rho_\mu(y)}{|x-y|} \, \mathrm{d} x \mathrm{d} y.
\end{align}
Hence,
\begin{align}\label{eq3.6v51}
J=J(\rho_\mu)= \frac{  \left(\int_{\mathbb{R}^3} \rho_\mu \,\mathrm{d}x  \right)^\frac23}{6K} .
\end{align}
For $\gamma=\frac{4}{3}$, the mass of $\rho_\mu$ is independent of $\mu$ which is denoted by $M_{K}$ (see Lemma 3.6 in \cite{ZC2020}). Thus, together with \eqref{eq3.6v51}, we have
\begin{align*}
J = \frac{ M_{K}^\frac{2}{3}}{6K}.
\end{align*}
\end{proof}
In the following, we compare $M_{K}$ with different critical masses used in  \cite{GLD,YTTZ,CS}.

First, we compare $M_{K}$ with the critical mass $M_0$ in \cite{CS}. For the Lane-Emden equation \eqref{ODE}, let
\begin{align*}
y(r)=\alpha\theta (\alpha \beta r)=\alpha\theta(s),\ s=\alpha\beta r,
\end{align*}
where $\alpha=4K\mu^\frac{1}{3}, \beta=\sqrt{\frac{4\pi}{(4K)^3}}$. Then, there holds
\begin{align*}
\theta''+\frac{2}{s}\theta'+ \theta_+^3=0,\ \theta(0)=1,\theta'(0)=0.
\end{align*}
which coincides with (3.1)-(3.3) for $\lambda=0$ in \cite{CS}.

Therefore, we have
\begin{align*}
M_{K}=\int_{\mathbb{R}^3}\rho_\mu   \, \mathrm{d} x
=\int_{\mathbb{R}^3}\Big(\frac{y(|x|)}{4K}\Big)^3 \, \mathrm{d} x
=\Big(\frac{K}{\pi}\Big)^\frac{3}{2}\int_{\mathbb{R}^3}\theta^3(|x|) \, \mathrm{d} x =  M_0.
\end{align*}
In Theorem 1.5 of \cite{YTTZ},
the authors proved that the star cannot collapse to a point when the mass is less than a ``critical'' mass
\begin{align}\label{eq3.9v63}
M_c  = \left(\frac{3K}{2\pi} \right)^\frac{3}{2} \left(\mathcal{M}_\frac{4}{3} \right)^{-2},
\end{align}
where $\mathcal{M}_\frac{4}{3}$ is the Marcinkiewicz interpolation constant of the following Marcinkiewicz interpolation inequality
\begin{align*}
\left( \int_{\mathbb{R}^3} \left|\hat{f} \right|^\frac{4}{3}|\xi|^{-2}\,  \mathrm{d} \xi \right)^\frac34
\leq \mathcal{M}_\frac{4}{3} \left( \int_{\mathbb{R}^3} |f|^\frac{4}{3} \, \mathrm{d} x \right)^\frac34.
\end{align*}
We have $M_c < M_{K}$. Indeed, in the proof of Theorem 1.5 in \cite{YTTZ}, they used the following estimate
\begin{align*}
&\iint_{\mathbb{R}^3\times\mathbb{R}^3}\frac{\rho(x)\rho(y)}{|x-y|} \,  \mathrm{d} x \mathrm{d} y
=4\pi \int_{\mathbb{R}^3}\frac{|\hat{\rho} (\xi) |^2}{|\xi|^2} \,  \mathrm{d} \xi \\
& \leq 4\pi  \left(\max_\xi \left| \hat{\rho} (\xi) \right|  \right)^\frac{2}{3}
\int_{\mathbb{R}^3}\frac{|\hat{\rho}|^\frac{4}{3}}{|\xi|^2} \, \mathrm{d} \xi
\leq 4\pi \mathcal{M}_\frac{4}{3}^\frac{4}{3}
\left(\int_{\mathbb{R}^3}\rho  \, \mathrm{d} x \right)^\frac{2}{3}\int_{\mathbb{R}^3}\rho^\frac{4}{3} \,  \mathrm{d} x.
\end{align*}
It is obvious that $4\pi\mathcal{M}^\frac{4}{3}_\frac{4}{3}$ is not optimal. Thus,
$4\pi\mathcal{M}_\frac{4}{3}^\frac{4}{3}>C_{\min}$ which implies $M_{K}>M_c$.

In \cite{GLD}, the authors proved the global existence of the spherically symmetric  weak solution of \eqref{1.1}-\eqref{eq1.5v34} for $\gamma = \frac43$ when the mass is less than a ``critical" mass $M_c \left( \frac43 \right)= \left(\frac{8\pi}{9K} \right)^{-\frac{3}{2}} \omega_4$, where $ \omega_4=\frac{2\pi^2}{\Gamma(2)}$.
Let's compare $M_{K}$ with $M_c \left(\frac43 \right)$.
In \cite{GLD}, they used the following inequality to obtain $M_c \left(\frac43 \right)$:
\begin{align*}
\iint_{\mathbb{R}^3\times \mathbb{R}^3}\frac{\rho(x)\rho(y)}{|x-y|}\,  \mathrm{d} x \mathrm{d} y
=\frac{1}{4\pi} \|\nabla V \|^2_{L^2 \left(\mathbb{R}^3 \right)}
&\leq \frac{16\pi \omega_4^{-\frac{2}{3}}}{3} ||\rho||^2_{L^\frac{6}{5} \left(\mathbb{R}^3 \right)}
\leq \frac{16\pi \omega_4^{-\frac{2}{3}}}{3}  ||\rho||_{L^1 \left(\mathbb{R}^3 \right)}^\frac{2}{3}||\rho||_{L^\frac{4}{3} \left(\mathbb{R}^3 \right)}^{\frac{4}{3}}.
\end{align*}
It is obvious that $\rho_\mu$ and $V(\rho_\mu)$ cannot make H\"{o}lder's inequality  to be an equality.
Hence, $C_{\min}<\frac{16\pi \omega_4^{-\frac{2}{3}}}{3}$ which yields $M_{K}> M_c \left(\frac{4}{3} \right)$.

We sum up the above discussion  in the following proposition.
\begin{proposition}\label{pro3.1}

\begin{enumerate}
\item  $M_{K}=M_0$ which is the critical mass in \cite{CS}.
\item  $M_{K}>M_c$ where $M_c = \left(\frac{3K}{2\pi} \right)^\frac{3}{2} \left(\mathcal{M}_\frac{4}{3} \right)^{-2}$
is the ``critical" mass in \cite{YTTZ}.
\item  $M_{K}> M_c \left(\frac43  \right)$ where $M_c \left(\frac43 \right)= \left(\frac{8\pi}{9K} \right)^{-\frac{2}{3}} \omega_4$ is the ``critical" mass in \cite{GLD} for $n=3$ and $\gamma=\frac{4}{3}$.
\end{enumerate}
\end{proposition}

\subsection{Existence of spherically symmetric global weak solutions when $\gamma = \frac43$}\label{subse4.1v28}
In this subsection, we show the global existence of a spherically symmetric finite energy  weak solution of the compressible Euler-Poisson equations when $\gamma = \frac43$. In the following, we first present the general framework in \cite{GLD} to show the existence of a global weak solution of Euler-Poisson equations when $\gamma \in \big( \frac65, \frac43 \big]$.
\subsubsection{Spherically symmetric finite energy weak solutions}\label{subsubse3.2.1v60}
Consider the Euler-Poisson equations
\begin{align}\label{4.1}
&\partial_t\rho+\mathrm{div}\mathcal{M}=0,\\
&\partial_t\mathcal{M}+\mathrm{div}\Big(\frac{\mathcal{M}\otimes\mathcal{M}}{\rho}\Big)+\nabla P+\rho\nabla V=0,\\
&\Delta  V=4\pi\rho,\\
&\lim_{|x|\rightarrow \infty}V(t, x)=0, \label{4.4}\\
& (\rho,\mathcal{M})|_{t=0}= \left(\rho_0,\mathcal{M}_0 \right), \label{eq4.5v47}
\end{align}
for $t>0$, $x\in\mathbb{R}^3$. Here, $P(\rho)=K\rho^\gamma$ is the pressure, $\mathcal{M} = \rho u  \in\mathbb{R}^3$ represents the momentum.
 Assume that the initial data has finite kinetic  energy and internal energy
\begin{align}\label{4.6}
E_0:=\int_{\mathbb{R}^3}\left(\frac{1}{2}  \left|\frac{\mathcal{M}_0}{\sqrt{\rho_0}} \right|^2+\frac{K\rho_0^\gamma}{\gamma-1}\right) \, \mathrm{d} x <\infty,
\end{align}
and finite mass
\begin{align}\label{4.7}
M(\rho_0)=\int_{\mathbb{R}^3}\rho_0 \, \mathrm{d} x <\infty.
\end{align}
\begin{definition}[Finite energy weak solution]\label{def4.7}
A measurable vector function $(\rho,\mathcal{M},V)$ is said to be a  spherically symmetric finite energy weak solution of the Cauchy problem \eqref{4.1}-\eqref{eq4.5v47} with  spherically symmetric initial data $\left(\rho_0(|x|),\mathcal{M}_0(|x|) \right)$ if the following conditions hold:
\begin{enumerate}
\item $\rho(t,|x|)\geq 0$ a.e., and $ \left(\mathcal{M},\frac{\mathcal{M}}{\sqrt{\rho}} \right)(t,|x|)=0$ a.e. on the vacuum states
$\{(t,x): \rho(t,|x|)=0\}$.
\item  For a.e. $t>0$, the energy is finite:
\begin{align*}
&\int_{\mathbb{R}^3}\left(\frac{1}{2} \left|\frac{\mathcal{M}}{\sqrt{\rho}} \right|^2
+\frac{K\rho^\gamma}{\gamma-1}+\frac{1}{ 8 \pi }|\nabla V|^2 \right)\,  \mathrm{d} x\leq C(E_0,M(\rho_0)),\\
&\int_{\mathbb{R}^3}\left(\frac{1}{2} \left|\frac{\mathcal{M}}{\sqrt{\rho}} \right|^2
+\frac{K\rho^\gamma}{\gamma-1}-\frac{1}{8 \pi }|\nabla V|^2\right) \, \mathrm{d} x
\leq \int_{\mathbb{R}^3}\left(\frac{1}{2 } \left|\frac{\mathcal{M}_0}{\sqrt{\rho_0}} \right|^2
+\frac{K\rho_0^\gamma}{\gamma-1}-\frac{1}{ 8 \pi }|\nabla V_0|^2\right) \, \mathrm{d} x.
\end{align*}
\item  For any $\phi(t,x)\in C_0^1(\mathbb{R}_+\times\mathbb{R}^3)$,
\begin{align*}
\int_{\mathbb{R}_+\times\mathbb{R}^3} \left(\rho\phi_t+\mathcal{M}\cdot\nabla\phi \right) \, \mathrm{d} x \mathrm{d}  t
+\int_{\mathbb{R}^3}\rho_0\phi(0,x) \, \mathrm{d} x=0.
\end{align*}
\item For any $\psi(t,x)\in \Big(C_0^1(\mathbb{R}_+ \times\mathbb{R}^3)\Big)^3$,
\begin{align*}
&\int_{\mathbb{R}_+ \times\mathbb{R}^3}\left(\mathcal{M}\cdot\partial_t\psi
+\frac{\mathcal{M}}{\sqrt{\rho}}\cdot \left(\frac{\mathcal{M}}{\sqrt{\rho}}\cdot\nabla \right)\psi+P(\rho)\mathrm{div}\psi\right) \, \mathrm{d} x \mathrm{d} t
+\int_{\mathbb{R}^3}\mathcal{M}_0\psi(0,x)\, \mathrm{d} x
\\&=\int_{\mathbb{R}_+\times\mathbb{R}^3}\rho\nabla V\cdot \psi  \, \mathrm{d} x \mathrm{d} t.
\end{align*}
\item  For $\varphi(x)\in C_0^1(\mathbb{R}^3)$,
\begin{align*}
\int_{\mathbb{R}^3}\nabla V\cdot\nabla\varphi  \, \mathrm{d} x
=-4\pi\int_{\mathbb{R}^3}\rho \varphi  \, \mathrm{d} x,\ \textrm{a.e.}\ t\geq 0.
\end{align*}
\end{enumerate}
\end{definition}
To establish the global existence of a finite energy weak solution,  adopting similar ideas of \cite{GLD}, we will use approximate solutions of the  compressible Navier-Stokes-Poisson equations to get the weak solutions of the compressible Euler-Poisson equations \eqref{4.1}-\eqref{4.4} by the vanishing viscosity limit argument.
To do this, we need the uniform estimates  and also the $H_{loc}^{-1}$ compactness of the entropy dissipation measures of the solutions to the following  compressible Navier-Stokes-Poisson equations:
\begin{align}
&\partial_t\rho+\text{div}\mathcal{M}=0,\label{4.14}\\
&\partial_t\mathcal{M}+\text{div} \Big(\frac{\mathcal{M}\otimes\mathcal{M}}{\rho} \Big)+\nabla P+\rho\nabla V
=\epsilon\text{div}\Big( \rho D \Big(\frac{\mathcal{M}}{\rho} \Big)\Big), \\
&\Delta  V=4\pi\rho,\label{4.16}\\
& \lim\limits_{|x| \to \infty} V(t,x) = 0, \label{eq3.20v74}
\end{align}
where $D(\frac{\mathcal{M}}{\rho})=\frac{1}{2}\left(\nabla \left(\frac{\mathcal{M}}{\rho} \right)
+ \left( \nabla \left(\frac{\mathcal{M}}{\rho} \right) \right)^T \right)$ is the stress tensor, and $P( \rho)   = K \rho^\gamma$.

To obtain the global finite energy weak solutions of  the compressible Navier-Stokes-Poisson equations \eqref{4.14}-\eqref{eq3.20v74},
for any $T> 0$, we will consider the following  free boundary problem of \eqref{4.14}-\eqref{eq3.20v74}:
\begin{align}\label{4.18}
&\partial_t\rho+\partial_r(\rho u)+\frac{2}{r}\rho u=0,\\
&\partial_t(\rho u)+\frac{\partial_r \big(r^2\rho u^2\big)}{r^2} + \partial_r P(\rho)
+\frac{4\pi \rho}{r^2}\int^r_a\rho(t,z)z^2 \, \mathrm{d} z
=\epsilon\partial_r\Big(\rho \Big(\partial_r u+\frac{2}{r}u \Big)\Big)-\epsilon\frac{2}{r} u\partial_r\rho,\label{4.19}
\end{align}
for $(t, r) \in \Omega_T$ with
\begin{align*}
\Omega_T = \left\{ (t,r): a \le r \le b(t), 0 \le t \le T\right\}.
\end{align*}
Here, $b(t)$ is the free boundary which is determined by
\begin{align*}
\begin{cases}
b'(t)=u(t,b(t)),\ 0\leq t\leq T,\\
b(0)=b,
\end{cases}
\end{align*}
and we  take $a=\frac{1}{b}$ with $b>0$ large enough. On the free boundary, we impose the stress free boundary condition:
\begin{align}\label{eq3.20v58}
\left(P-\epsilon\rho \left(\partial_ru+\frac{2}{r}u \right)\right)(t,b(t))=0, \text{ for } t \in (0, T]  .
\end{align}
On the fixed inner boundary $r=a$, the Dirichlet boundary condition is chosen:
\begin{align}\label{4.23}
u(t,a)=0, \text{ for } t \in (0, T] .
\end{align}
Once the   approximate initial data $ \left(\rho_0^{\epsilon,b},u_0^{\epsilon,b} \right)$ satisfies
\begin{align*}
0<C_{\epsilon,b}^{-1}\leq \rho_0^{\epsilon,b}(r)\leq C_{\epsilon,b}<\infty,
\end{align*}
the existence of global  approximate smooth solutions $ \left(\rho^{\epsilon,b},u^{\epsilon,b} \right)$ of \eqref{4.18}-\eqref{4.23}
 with $0<\rho^{\epsilon,b}(t,r)<\infty$ is obtained.
 We can extend $\rho^{\epsilon,b}$ to be zero outside $\Omega_t= \left\{x\in\mathbb{R}^3 : a\leq|x|\leq b(t) \right\}$.
 Define the potential function $V$ to be the solution of the following Poisson equation:
\begin{align*}
\Delta  V=4\pi\rho^{\epsilon,b},\ \lim_{|x|\rightarrow\infty} V(t,x) =0.
\end{align*}
 We can see
\begin{align*}
V_{r}(t,r)=\left\{
             \begin{array}{ll}
               0, & 0\leq r< a, \\
               \\
               \frac{4\pi}{r^2}\int^r_0\rho^{\epsilon,b}(t,s)  s^2 \, \mathrm{d} s, & a\leq r\leq b(t), \\
               \\
               \frac{M(\rho_0)}{r^2}, & r>b(t).
             \end{array}
           \right.
\end{align*}
Indeed, the global existence of smooth solutions of the  problem \eqref{4.18}-\eqref{4.23} can be obtained by using the argument in Section 3 of \cite{DL}. An important step in the proof is to obtain the boundedness of $L^\gamma(\mathbb{R}^3)$-norm (the internal energy).
Then, we can obtain a global finite energy weak solution of problem \eqref{4.1}-\eqref{4.4} by taking $b\rightarrow \infty$ and $\epsilon\rightarrow 0$ by similar arguments in Sections 4  and 5 of \cite{GLD}.

\subsubsection{Existence of spherically symmetric global weak solutions when $\gamma = \frac43$}
In this subsection, we show the global existence of a spherically symmetric finite energy weak solution of the compressible Euler-Poisson equations when $\gamma = \frac43$ by the framework in \cite{GLD}.
We present the key steps to recover the estimates in \cite{GLD}. Define the following important quantities:
\begin{align*}
M(\rho^{\epsilon,b})&=\int_{a<|x|<  b(t)}\rho^{\epsilon,b}(x)  \, \mathrm{d} x
=4\pi\int^{b(t)}_a\rho^{\epsilon,b}(r)r^2 \, \mathrm{d} r,\\
E(\rho^{\epsilon,b},u^{\epsilon,b})
&=\frac{1}{2}\int_{a<|x|<b(t)}\rho^{\epsilon,b}|u^{\epsilon,b}|^2 \, \mathrm{d} x
+3\int_{a<|x|<b(t)} K(\rho^{\epsilon,b})^\frac{4}{3} \,  \mathrm{d} x
-\frac{1}{8\pi} \int_{ \mathbb{R}^3}  
 \left|\nabla V(\rho^{\epsilon,b}) \right|^2 \,  \mathrm{d} x\\
&=2\pi\int_a^{b(t)}\rho^{\epsilon,b}|u^{\epsilon,b}|^2r^2 \, \mathrm{d} r+12\pi\int_{a}^{b(t)} K(\rho^{\epsilon,b})^\frac{4}{3}r^2 \, \mathrm{d} r
-\frac{1}{2}\int_0^{ \infty} \left|\partial_r V(\rho^{\epsilon,b}) \right|^2r^2 \, \mathrm{d} r\\
&=2\pi\int_a^{b(t)}\rho^{\epsilon,b}|u^{\epsilon,b}|^2r^2  \, \mathrm{d} r
+12\pi\int_{a}^{b(t)} K(\rho^{\epsilon,b})^\frac{4}{3}r^2 \, \mathrm{d} r
-8\pi^2\int_0^{ \infty}\frac{1}{r^2} \left(\int^r_a\rho^{\epsilon,b} z^2  \, \mathrm{d} z \right)^2 \, \mathrm{d} r,
\end{align*}
where $\rho^{\epsilon,b}(t,r)$ is understood to be 0 for $r \in [0, a] \cup (b(t), \infty)$.

By Lemma A.10 in \cite{GLD}, the approximate initial data $(\rho_0^{\epsilon,b},u_0^{\epsilon,b})$ satisfies
\begin{align*}
M \big(\rho_0^{\epsilon,b} \big)= M(\rho_0),
\end{align*}
and
\begin{align*}
&2\pi\int_a^{b(t)}\rho_0^{\epsilon,b} \left|u_0^{\epsilon,b} \right|^2r^2 \, \mathrm{d} r
+12\pi\int_{a}^{b(t)}K \big(\rho_0^{\epsilon,b} \big)^\frac{4}{3}r^2 \, \mathrm{d} r
\rightarrow2\pi\int_0^{ \infty}\rho_0|u_0|^2r^2\, \mathrm{d} r
+12\pi\int_0^{ \infty}K\rho_0^\frac{4}{3}r^2 \, \mathrm{d} r,\\
&\frac{1}{2}\int_0^{\infty} \left|\partial_rV^{\epsilon,b}_0 \right|^2r^2 \, \mathrm{d} r
\rightarrow \frac{1}{2}\int_0^{ \infty} \left|\partial_rV_0  \right|^2r^2\,  \mathrm{d} r,
\end{align*}
by first taking $b\rightarrow\infty$ and then $\epsilon\rightarrow 0$. Here $\Delta  V_0=4\pi\rho_0$ and $\Delta V^{\epsilon,b}_0=4\pi\rho^{\epsilon,b}_0$.

Thanks to  $\rho^{\epsilon,b}=0$ for $r\in[0,a)\cup(b(t),\infty)$, by \eqref{2.9}, we have
\begin{align*}
E \left(\rho^{\epsilon,b}, u^{\epsilon,b} \right)
\geq  2\pi\int_0^{ \infty}\rho^{\epsilon,b} \left|u^{\epsilon,b} \right|^2r^2 \, \mathrm{d} r
+ \left(3K-\frac{1}{2}C_{min}\left(M \left(\rho_0^{\epsilon,b} \right)\right)^\frac{2}{3}\right)
\int_{\mathbb{R}^3} \left(\rho^{\epsilon,b} \right)^\frac{4}{3} \, \mathrm{d} x.
\end{align*}
If the original initial data $(\rho_0,u_0)$ satisfies $M(\rho_0)< M_{K}$,
then the approximate initial data $ \left(\rho_0^{\epsilon,b},u_0^{\epsilon,b} \right)$  satisfies $M \big(\rho_0^{\epsilon,b} \big) = M(\rho_0)<M_{K}$. It implies that  $||\rho^{\epsilon,b}||_{L^\frac{4}{3}(\mathbb{R}^3)}\leq C(M(\rho_0),E_0)$.

Furthermore, by \eqref{2.9} and $M(\rho_0)<M_{K}$, we also have
$$\iint_{\mathbb{R}^3\times\mathbb{R}^3}\frac{\rho^{\epsilon,b}(x)\rho^{\epsilon,b}(y)}{|x-y|}\, \mathrm{d} x \mathrm{d} y\leq C (M(\rho_0),E_0).$$
The rest of the estimates are very similar to \cite{GLD}, and we skip. Hence, the global finite energy weak solution of  the Euler-Poisson equations is obtained by first taking $b\rightarrow\infty$ and then taking the vanishing viscosity limit $\epsilon\rightarrow0$ by similar arguments in \cite{GLD}.

Moreover, one can solve the Euler-Poisson equations  for $t<0$ when the spherically symmetric initial data $(\rho_0,\frac{\mathcal{M}_0}{\sqrt{\rho_0}})$ satisfies \eqref{4.6}-\eqref{4.7} and the total mass  $M(\rho_0)$ satisfies $M(\rho_0) <M_{K}$. Indeed, there exists the global finite energy weak solution $(\tilde{\rho}(t,x),\tilde{u}(t,x))=\left(\tilde{\rho},\frac{\widetilde{\mathcal{M}}}{\sqrt{\tilde{\rho}}} \right)$  of  \eqref{1.1}-\eqref{eq1.5v34} for $t>0$ with the initial spherically symmetric data $(\rho_0,-\frac{\mathcal{M}_0}{\sqrt{\rho_0}})$. Then, $(\rho(t,x),u(t,x))=(\tilde{\rho}(-t,x),-\tilde{u}(-t,x))$ is a global finite energy weak solution of  \eqref{1.1}-\eqref{eq1.5v34} for $t<0$ with the initial data $(\rho_0,\frac{\mathcal{M}_0}{\sqrt{\rho_0}})$.

In conclusion, the main result in this part is
\begin{theorem}
When $\gamma=\frac{4}{3}$,
assume that the initial spherically symmetric data $(\rho_0,\mathcal{M}_0,V_0)$ satisfies \eqref{4.6} and $M(\rho_0) <M_{K}$.
Then, there exists a global finite energy weak solution $(\rho,\mathcal{M},V)$ with spherical symmetry for both $t>0$ and $t<0$.

\end{theorem}

\subsection{Expanding strong solutions}
 Assume that the spherically symmetric global strong solution of  \eqref{1.5}-\eqref{1.8} exists  when the mass is less than $M_{K}$.
  By the law of energy conservation and \eqref{2.9}, we have
\begin{align}\label{eq4.13v34}
E(\rho_0,u_0)=E(\rho,u)
\geq \left(3K-\frac{1}{2}C_{\min}(M(\rho_0))^\frac{2}{3}\right)\int_{\mathbb{R}^3}\rho^\frac{4}{3} \, \mathrm{d} x>0,
\end{align}
where we used the optimal Hardy-Littlewood inequality in Theorem \ref{thm3.1}.

Then, from
\begin{align*}
M(\rho_0)=\int_{\Omega(t)}\rho\, \mathrm{d}x\leq |\Omega(t)|^\frac{1}{4}\left(\int_{\Omega(t)}\rho^\frac{4}{3}\, \mathrm{d}x\right)^\frac{3}{4}\leq |\Omega(t)|^\frac{1}{4}\left(\frac{E(\rho_0,u_0)}{3K-\frac{1}{2}C_{\min}(M(\rho_0))^\frac{2}{3}}\right)^\frac{3}{4},
\end{align*}
we have
\begin{align*}
\frac{4\pi}{3}R^3(t)\geq |\Omega(t)|\geq M(\rho_0)^4\left(\frac{3K-\frac{1}{2}C_{\min}(M(\rho_0))^\frac{2}{3}}{E(\rho_0,u_0)}\right)^3>0.
\end{align*}
It shows that the measure of the support is bounded below.

Moreover, let
\begin{align*}
H(t)=\frac{1}{2}\int_{\mathbb{R}^3}\rho|x|^2 \, \mathrm{d} x.
\end{align*}
By direct calculation and \eqref{eq4.13v34}, we have
\begin{align*}
H'(t)&=\int_{\mathbb{R}^3}\rho u\cdot x \, \mathrm{d} x, \\
H''(t)&= E(\rho,u)+\frac{1}{2}\int_{\mathbb{R}^3}\rho|u|^2 \, \mathrm{d} x \geq E(\rho_0,u_0)>0.
\end{align*}
Since $R(t)$ is the moving interface of fluids and vacuum states, we get
\begin{align*}
H(t)\leq \frac{1}{2}M(\rho_0)R^2(t).
\end{align*}
Therefore, by the fundamental theorem of calculus, we have
\begin{align*}
\frac{1}{2} M(\rho_0)R^2(t)\geq \frac{1}{2}E(\rho_0,u_0) t^2+H'(0)t+H(0),
\end{align*}
which implies \eqref{1.14} for both $t\in [0,+\infty)$ and $t\in(-\infty,0]$.

\begin{remark}
In \cite{YTTZ,MP}, it was shown that initial data with positive energy leads to expanding solutions assuming the global existence of strong solutions. However, the results in \cite{CS} shows that the positive energy does not even imply the global existence of weak solutions. Indeed, as a counterexample, there exists self-simialr collapsing solutions with positive energy (\cite{CS}). However, by Theorem \ref{th1.1},  the condition that the total mass is less than $M_{K}$ not only implies the global existence of weak solution, but also shows that the solution cannot collapse (i.e. measure of the support is bounded below) and the  diameter of support tends to infinity as time goes to both positive and negative infinity. Moreover, the critical mass $M_{K}$ is sharp for expanding solutions in the sense that there exists self-similar collapsing solutions (\cite{CS} \cite{M2}) with mass greater than or equal to $M_{K}$.
\end{remark}

\section{ Expanding solutions for polytropes with $\gamma \in \left( \frac65, \frac43 \right) $ }
\label{se4v28}
In this section, we study the polytrope for $\gamma\in \left(\frac{6}{5},\frac{4}{3} \right)$.

\subsection{A constrained variational problem}\label{subse3.2v31}

We consider the following constrained variational problem:
\begin{align}\label{3.61}
l_\mu:=\inf\limits_{ \rho \in \mathcal{K}} S_\mu ( \rho),
\end{align}
where $\mathcal{K}$ is defined by \eqref{K}.

For $\rho \in \mathcal{K}$, we have
\begin{align}\label{3.62}
S_\mu (\rho)
=&\frac{4-3\gamma}{\gamma-1}\int_{\mathbb{R}^3}K\rho^\gamma \,  \mathrm{d} x
- V_\mu(R_\mu )  \int_{\mathbb{R}^3} \rho \, \mathrm{d} x\geq 0.
\end{align}
Therefore,
\begin{align*}
l_\mu  =\inf_{\rho\in\mathcal{K}}  S_\mu (\rho)  \geq 0.
\end{align*}

Let $\rho_\lambda(x)=\lambda^3\rho(\lambda x)$, $\lambda>0$. Then,
\begin{align*}
&M(\rho_\lambda) =M(\rho),
\intertext{ and }
&Q(\rho_\lambda)=\lambda^{3\gamma-3} \int_{\mathbb{R}^3}3K\rho^\gamma \, \mathrm{d} x
-\frac{\lambda}{2}\iint_{\mathbb{R}^3\times\mathbb{R}^3}\frac{\rho(x)\rho(y)}{|x-y|} \, \mathrm{d} x \mathrm{d} y.
\end{align*}
For any $\rho\not\equiv  0$, there exists a unique
\begin{align}\label{3.56}
\lambda^*(\rho)
= \left( \frac{6\int_{\mathbb{R}^3}K\rho^\gamma  \, \mathrm{d} x}
{\iint_{\mathbb{R}^3\times\mathbb{R}^3}\frac{\rho(x)\rho(y)}{|x-y|} \,\mathrm{d} x \mathrm{d} y} \right)^\frac1{4-3\gamma}
\end{align}
such that $\rho_{\lambda^*(\rho)}\in \mathcal{K}$. We also have $\lambda^*(\rho)<1$ if and only if $Q(\rho)<0$,   $\lambda^*(\rho)=1$ if and only if $Q(\rho)=0$.
Direct calculation yields for any $\lambda > 0$,
\begin{align}
&S_{\mu,\lambda}=S_\mu(\rho_\lambda)
=\lambda^{3\gamma-3}\int_{\mathbb{R}^3}\frac{K}{\gamma-1}\rho^\gamma  \,  \mathrm{d} x
-\frac{\lambda}{2}\iint_{\mathbb{R}^3\times\mathbb{R}^3}\frac{\rho(x)\rho(y)}{|x-y|} \,  \mathrm{d} x \mathrm{d} y
- V_\mu(R_\mu)\int_{\mathbb{R}^3}\rho \,  \mathrm{d} x,  \notag \\
&\frac{ \mathrm{d} S_{\mu,\lambda}}{ \mathrm{d} \lambda}
= \frac1\lambda Q( \rho_\lambda) , \label{3.58}
\intertext{ and }
&\frac{ \mathrm{d}^2S_{\mu,\lambda}}{ \mathrm{d} \lambda^2}
=3(3\gamma-4)\lambda^{3\gamma-5}\int_{\mathbb{R}^3}K\rho^\gamma \, \mathrm{d} x. \notag
\end{align}

Let $\rho_{\mu}$ be the density of the Lane-Emden star with the center density
$\mu$. Then $S_{\mu}^{\prime}\left(  \rho_{\mu}\right)  =0$. Choosing
$\rho=\rho_{\mu}$ and $\lambda=1$ in \eqref{3.58}, we have $Q\left(  \rho_{\mu}\right)
=0$ which is a Pohozaev type identity.

We summarize some facts to be used later.
\begin{lemma}\label{lem3.2}
For $\rho\in L^\gamma(\mathbb{R}^3)$, $\rho\geq 0,$ $\rho \not\equiv 0$, we have
\begin{enumerate}
\item[(i)] There exists $\lambda^*(\rho) > 0 $ such that $\rho_{\lambda^*(\rho)}\in \mathcal{K}$.
\item[(ii)] $S_{\mu,\lambda}$ is concave with respect to $\lambda\in (0,  \infty)$.
\item[(iii)]  $\lambda^* ( \rho) <1\Leftrightarrow Q(\rho)<0$.
\item[(iv)]   $\lambda^* ( \rho) =1 \Leftrightarrow Q(\rho)=0$.
\item[(v)]   $S_\mu (\rho_\lambda)<S_\mu  \left(\rho_{\lambda^*} \right)$ for any $\lambda>0,\lambda\neq \lambda^*$.
\item[(vi)]  If $\rho_n\rightharpoonup\bar{\rho}$ weakly in $L^\gamma(\mathbb{R}^3)$,
 we have $\rho_{n,\lambda} (x) :=\lambda^3\rho_n(\lambda x)\rightharpoonup\bar{\rho}_\lambda(x) $ weakly in $L^\gamma(\mathbb{R}^3)$ as $n\rightarrow\infty$.
\item[(vii)]  $ \left(\rho_\lambda \right)^\sharp =  \left( \rho^\sharp \right)_\lambda $ for any $\lambda > 0$, where $ (\cdot)^\sharp$ is the Schwartz symmetrization.
\end{enumerate}
\end{lemma}

\begin{theorem}\label{thm3.4}
$\tilde{\rho}$ attains  $ \min\limits_{\rho\in \mathcal{K}} S_\mu (\rho)$ if and only if $\tilde{\rho}$ is some translation of $\rho_\mu$.
\end{theorem}
\begin{proof}
\textbf{Step 1.}
 Let $\{\rho_n\}$ be a minimizing sequence of \eqref{3.61} such that
 \begin{align}
 Q(\rho_n)=0 \label{eq3.14v34}
\intertext{ and }
\lim\limits_{n\rightarrow\infty}S_\mu (\rho_n)=l_\mu .\notag
\end{align}
For the minimizing sequence $\{\rho_n\}$, we can define $w_n = \rho_n^\sharp$ which is the Schwartz symmetric rearrangement of $\rho_n$ and $u_n = w_{n, \lambda^*(w_n)}$. It follows from Lemma \ref{lem3.2} $(i)$ that $u_n \in \mathcal{K} $. Furthermore, it follows from Lemma \ref{lem3.2} (vii) that
$u_n = \left( \rho_{n, \lambda^*(w_n)} \right)^\sharp$. Therefore, by (i) and (v) of Lemma \ref{lem3.2}, we have
\begin{align*}
S_{\mu}(u_{n})\leq S_{\mu}\left(  \rho_{n,\lambda^{\ast}(w_{n})}\right)  \leq
S_{\mu}\left(  \rho_{n,\lambda^{\ast}(\rho_{n})}\right)  =S_{\mu}(\rho_{n}).
\end{align*}
Hence, we can assume $\{ \rho_n \}$ is a non-negative, spherically symmetric and non-increasing sequence.

From \eqref{3.62}, there holds
\begin{align}\label{3.65}
\|\rho_n\|_{L^\gamma(\mathbb{R}^3)}+ \|\rho_n\|_{L^1 \left(\mathbb{R}^3 \right)} \lesssim 1.
\end{align}
Denote $\triangle V_n=4\pi\rho_n$. By \eqref{2.8}, we have
\begin{align}\label{eq4.6v59}
\|\nabla V_n\|^2_{L^2 \left(\mathbb{R}^3 \right)} \lesssim
 \|\rho_n\|_{L^1}^{\frac{5\gamma-6}{3(\gamma-1)}} \|\rho_n\|_{L^\gamma}^{\frac{\gamma}{3(\gamma-1)}}.
\end{align}
From \eqref{eq3.14v34}, we have
\begin{align}\label{3.67}
3\int_{\mathbb{R}^3}K\rho_n^\gamma \, \mathrm{d}  x-\frac{1}{8\pi}\int_{\mathbb{R}^3}|\nabla V_n|^2\,  \mathrm{d} x=0.
\end{align}
It follows from \eqref{eq4.6v59} and \eqref{3.67} that
\begin{align}\label{eq3.17v44}
\|\rho_n\|_{L^\gamma \left(\mathbb{R}^3 \right)}^\gamma= \frac{1}{24\pi} \|\nabla V_n\|_{L^2(\mathbb{R}^3)}^2 \lesssim
\|\rho_n\|_{L^1 \left(\mathbb{R}^3 \right)}^\frac{5\gamma-6}{3(\gamma-1)}  \|\rho_n\|_{L^\gamma(\mathbb{R}^3)}^\frac{\gamma}{3(\gamma-1)}.
\end{align}
This together with \eqref{3.65} yields
\begin{align}\label{3.69}
K^\frac{3( \gamma - 1)}{\gamma ( 4- 3 \gamma) } \lesssim \|\rho_n\|_{L^\gamma \left(\mathbb{R}^3 \right)}.
\end{align}
Combining with \eqref{3.62}, it implies $l_\mu  >0$.

\textbf{Step 2.}
From \eqref{3.65}, there exists $\bar{\rho}$ such that $\rho_n\rightharpoonup \bar{\rho}$ weakly in $L^\gamma_{rad} \left(\mathbb{R}^3 \right)$.
By Lemma  \ref{lem2.3}, we get
\begin{align}\label{eq5.10v56}
\nabla V(\rho_n)\rightarrow \nabla V(\bar{\rho}) \text{ strongly in $L^2_{rad} \left(\mathbb{R}^3 \right)$}, \text{ as } n \to \infty.
\end{align}
Meanwhile, from \eqref{3.67} and \eqref{3.69}, we have
\begin{align}\label{eq4.11v59}
\|\nabla V(\rho_n)\|_{L^2  \left(\mathbb{R}^3 \right)} \gtrsim K^\frac{3( \gamma -1)}{ 4 - 3 \gamma}.
\end{align}
From Lemma \ref{lem2.1}, we have $\bar{\rho}\geq 0$, $\bar{\rho}\not\equiv 0$  and
\begin{align*}
\|\bar{\rho}\|_{L^\gamma  \left(\mathbb{R}^3 \right)}+ \|\bar{\rho}\|_{L^1  \left(\mathbb{R}^3 \right)} \lesssim 1.
\end{align*}
By \eqref{eq5.10v56} and \eqref{eq4.11v59}, we have
\begin{align*}
 \left\|\nabla V(\bar{\rho}) \right\|_{L^2 \left(\mathbb{R}^3 \right)}
\gtrsim  K^\frac{3( \gamma - 1)}{ 4 - 3 \gamma}.
\end{align*}
Thus, $\bar{\rho} \ne 0$.
By Lemma \ref{lem2.1} and the weakly lower semi-continuity of $\int_{\mathbb{R}^3}\rho^{\gamma}dx$, we get
\begin{align*}
\int_{\mathbb{R}^3}\bar{\rho} \,  \mathrm{d} x\leq \liminf_{n\rightarrow\infty}\int_{\mathbb{R}^3}\rho_n \,  \mathrm{d} x \text{ and }
\int_{\mathbb{R}^3}\bar{\rho}^\gamma \, \mathrm{d} x\leq \liminf_{n\rightarrow\infty}\int_{\mathbb{R}^3}\rho^\gamma_n \,  \mathrm{d} x.
\end{align*}
Above estimates together with \eqref{eq5.10v56} and \eqref{eq2.1v56} yield $S_\mu (\bar{\rho})\leq  \liminf\limits_{n\to \infty} S_\mu (\rho_n)$.

From Lemma \ref{lem3.2} (i), we have
\begin{align}\label{3.79}
\bar{\rho}^* :=\bar{\rho}_{\lambda^*(\bar{\rho})}\in \mathcal{K}
 \end{align}
and $\rho_{n,\lambda^*(\bar{\rho})}\rightharpoonup\bar{\rho}^*$ weakly in $L^\gamma \left(\mathbb{R}^3 \right)$, as $n\to \infty$.
By similar arguments as above, we have
\begin{align*}
S_\mu  \left(\bar{\rho}^* \right)\leq\liminf_{n\rightarrow\infty}S_\mu  \left(\rho_{n,\lambda^*(\bar{\rho})} \right).
\end{align*}
Moreover, from Lemma \ref{lem3.2} (iv) and (v), we have
\begin{align*}
\liminf_{n\rightarrow\infty}S_\mu  \left(\rho_{n,\lambda^* \left(\bar{\rho} \right)} \right)
\leq\liminf_{n\rightarrow\infty}S_\mu  \left(\rho_{n,\lambda^*(\rho_n)} \right)
=\liminf_{n\rightarrow\infty}S_\mu (\rho_{n})=l_\mu .
\end{align*}
Hence,
\begin{align}\label{3.76}
S_\mu  \left(\bar{\rho}^* \right)
=\min_{\rho\in \mathcal{K}} S_\mu (\rho).
\end{align}
This shows that $ \min\limits_{\rho\in \mathcal{K}} S_\mu (\rho)$ is obtained.

\textbf{Step 3.} Let $\bar{\rho}^{\ast}$ be a spherically symmetric minimizer \eqref{3.61}.
We will show $\bar{\rho}^{\ast}=\rho_{\mu}$.
First, we derive the Euler-Lagrange equation for $\bar{\rho}^*$. For $\epsilon>0$, define
\begin{align*}
\mathcal{S}_\epsilon:=\{x\in\mathbb{R}^3: \epsilon<\bar{\rho}^*<\frac{1}{\epsilon}\}.
\end{align*}
Choose a test function $w\in L^\infty(\mathbb{R}^3)$ which has compact support and is non-negative on $\mathbb{R}^3\setminus \mathcal{S}_\epsilon$. For $\tau\geq 0$ small enough, define
\begin{align*}
\rho^*_\tau=\bar{\rho}^*+\tau w.
\end{align*}
Then $\rho^*_\tau\geq 0$. By \eqref{3.56}, there exists a positive constant
\begin{align*}
\lambda_\tau
= \left( \frac{6\int_{\mathbb{R}^3}K(\rho^*_\tau)^\gamma  \, \mathrm{d} x}
{\iint_{\mathbb{R}^3\times\mathbb{R}^3}\frac{\rho^*_\tau(x)\rho^*_\tau(y)}{|x-y|} \,\mathrm{d} x \mathrm{d} y} \right)^\frac1{4-3\gamma}
\end{align*}
such that $\lambda^3_\tau\rho^*_\tau(\lambda_\tau x)\in\mathcal{K}$.

Since $\bar{\rho}^*$ is a minimizer of \eqref{3.61}, it follows that
\begin{align*}
0\leq& S_\mu(\lambda^3_\tau\rho^*_\tau(\lambda_\tau x))-S_\mu(\bar{\rho}^*)\\
=&\left(\int_{\mathbb{R}^3}(\frac{\gamma}{\gamma-1} K\left(\bar{\rho}^* \right)^{\gamma - 1}+ V \left(\bar{\rho}^* \right)- V_\mu(R_\mu))w\,\mathrm{d} x\right)\tau+o(\tau).
\end{align*}
Here, we use $Q(\bar{\rho}^*)=0$. Therefore, the  coefficient of $\tau$  must be non-negative. That is
\begin{align*}
\int_{\mathbb{R}^3}(\frac{\gamma}{\gamma-1} K\left(\bar{\rho}^* \right)^{\gamma - 1}+ V \left(\bar{\rho}^* \right)- V_\mu(R_\mu))w\,\mathrm{d} x\geq0.
\end{align*}
Hence, $\frac{\gamma}{\gamma-1} K\left(\bar{\rho}^* \right)^{\gamma - 1}+ V \left(\bar{\rho}^* \right)- V_\mu(R_\mu)\geq 0$ in $\mathcal{S}_\epsilon$ and $\frac{\gamma}{\gamma-1} K\left(\bar{\rho}^* \right)^{\gamma - 1}+ V \left(\bar{\rho}^* \right)- V_\mu(R_\mu)\geq 0,$ in $\mathbb{R}^3\setminus\mathcal{S}_\epsilon$ for all $\epsilon>0$ small enough. Taking $\epsilon\rightarrow0$, we get
\begin{align}\label{star}
\frac{\gamma}{\gamma-1} K\left(\bar{\rho}^* \right)^{\gamma - 1}+ V \left(\bar{\rho}^* \right)= V_\mu(R_\mu)=-\frac{M_1}{R_1}\mu^{\gamma-1},\ \text{in}\ \{x\in\mathbb{R}^3:\bar{\rho}^*>0\}.
\end{align}

Since the terms $\left(\bar{\rho}^* \right)^{\gamma - 1}$ and $V \left(\bar{\rho}^* \right)$ in \eqref{star} belong to $L^{\frac{\gamma}{\gamma-1}}_{rad}(\mathbb{R}^3)$, it obtains that the support of $\bar{\rho}^*$ is compact.
Note that $\bar{\rho}^*$ is spherically symmetric and non-increasing.
There exists $0<R_\mu<\infty$ such that $\bar{\rho}^*(|x|)=0$ when $x\in  B^c_{R_\mu} $.
From the regularity theory \cite{GT}, we have $ \left(\bar{\rho}^* \right)^{\gamma-1}\in C^2(B_{R_\mu} )$.
Let $F(s)= \left(\Phi' \right)^{-1}(s)$ for $s\in(0, \infty)$, where $\Phi$ is the enthalpy function.
We extend $F(s)$ to $(-\infty,0)$ by zero extension. This extended function is denoted by $F_+(s):\mathbb{R}\rightarrow[0, \infty)$. Therefore, the potential $V(\bar{\rho}^*)(r)=:V(r)$ satisfies
\begin{align*}
\Delta  V=V''+\frac{2}{r}V'=4\pi F_+\left(-\frac{M_1}{R_1}\mu^{\gamma-1}-V\right).
\end{align*}
Let $\bar{y}_\mu(r)=-\frac{M_1}{R_1}\mu^{\gamma-1}-V(r)= \frac\gamma{\gamma - 1} K (\bar{\rho}^*)^{\gamma - 1} =\Phi'(\bar{\rho}^*).$
Thus, we obtain the following ODE:
\begin{align*}
\begin{cases}
\bar{y}''_\mu+\frac{2}{r} \bar{y}'_\mu=-4\pi F_+( \bar{y}_\mu ),\\
\bar{y}_\mu (0)=\Phi'(\bar{\rho}^*(0)) ,\  \bar{y}'_\mu(0)=0.
\end{cases}
\end{align*}
This system is equivalent to
\begin{align*}
\bar{y}'_\mu(r)=-\frac{4\pi}{r^2}\int^r_0s^2F_+( \bar{y}_\mu (s)) \, \mathrm{d} s,\  \bar{y}_\mu (0)=\Phi'(\bar{\rho}^*(0)).
\end{align*}
From the classical ODE theory, this problem admits a unique solution $\bar{y}_\mu(r)$. Because $\bar{y}_\mu(R_\mu)= \frac\gamma{\gamma - 1} K (\bar{\rho}^*(R_\mu))^{\gamma - 1}=0$, there holds $-\frac{M_1}{R_1}\mu^{\gamma-1}=V(R_\mu)$. Note that  $\bar{\rho}^*$ is the solution of \eqref{1.11}  with  compact support $B_{R_\mu}$. 
We have $V(R_\mu)=-\frac{M_1}{R_1}(\bar{\rho}^*(0))^{\gamma-1}$ and consequently the center density $\bar{\rho}^*(0)=\mu$. Thus, $\bar{\rho}^{\ast}=\rho_{\mu}$.

\textbf{Step 4.}
Assume that $\tilde{\rho}$ is a minimizer of the variational problem
$\eqref{3.61}$. We will show that $\tilde{\rho}$ must be some translation of
$\rho_{\mu}$. Let $\tilde{\rho}^{\sharp}$ be the Schwartz symmetric
rearrangement of $\tilde{\rho}$. It follows from Lemma \ref{lem3.2} (i) that
$\tilde{\rho}_{\lambda^{\ast}(\tilde{\rho}^{\sharp})}^{\sharp}\in\mathcal{K}$.
Moreover, by Lemma \ref{lem3.2} (vii), $\tilde{\rho}_{\lambda^{\ast}%
(\tilde{\rho}^{\sharp})}^{\sharp}=\left(  \tilde{\rho}_{\lambda^{\ast}%
(\tilde{\rho}^{\sharp})}\right)  ^{\sharp}$. Therefore, by (i) and (v) of
Lemma \ref{lem3.2}, we have
\[
\min_{\rho\in\mathcal{K}%
}S_{\mu}(\rho)\leq S_{\mu}\left(  \tilde{\rho}_{\lambda^{\ast}(\tilde{\rho}^{\sharp})}^{\sharp
}\right)  \leq S_{\mu}\left(  \tilde{\rho}_{\lambda^{\ast}(\tilde{\rho
}^{\sharp})}\right)  \leq S_{\mu}\left(  \tilde{\rho}_{\lambda^{\ast}%
(\tilde{\rho})}\right)  =S_{\mu}(\tilde{\rho})=\min_{\rho\in\mathcal{K}%
}S_{\mu}(\rho).
\]
Therefore, $\lambda^{\ast}(\tilde{\rho}^{\sharp})=\lambda^{\ast}(\tilde{\rho
})=1$ and $S_{\mu}\left(  \tilde{\rho}^{\sharp}\right)  =S_{\mu}\left(
\tilde{\rho}\right)  $. Consequently, the equality in
(\ref{rearrangement-Gra-energy}) holds true for $\tilde{\rho}$. By Theorem 3.9
of \cite{EM}, $\tilde{\rho}$ must be spherically symmetric up to some
translation. Together with Step 3, this implies that $\tilde{\rho}$ is some
translation of $\rho_{\mu}$.
\end{proof}
\begin{remark}
Different from the case  $\gamma < \frac43$, the infimum  $\inf\limits_{ \rho \in \mathcal{K}} S_\mu ( \rho)$ is not bounded from below for the case $\frac43<\gamma<2$. Indeed, for $\lambda>0$ and fixed $\rho \in \mathcal{K}$, we have $\lambda^\frac2{2- \gamma} \rho( \lambda x)\in\mathcal{K}$ and
\begin{align*}
S_\mu \left(\lambda^\frac2{2- \gamma} \rho( \lambda x) \right) & = \frac{4 - 3 \gamma }{\gamma - 1} \int K  \left(\lambda^\frac2{2- \gamma} \rho( \lambda x) \right)^\gamma \,\mathrm{d}x - V_\mu (R_\mu) \int \lambda^\frac2{2- \gamma} \rho( \lambda x)\,\mathrm{d}x\\
& = \frac{4 - 3\gamma}{ \gamma - 1} \lambda^\frac{5 \gamma - 6}{2 - \gamma} \int K\rho^\gamma \,\mathrm{d}x - V_\mu (R_\mu) \lambda^\frac{3 \gamma - 4}{ 2- \gamma} \int \rho\,\mathrm{d}x \to - \infty, \text{ as } \lambda \to \infty.
\end{align*}
\end{remark}
\begin{lemma}\label{cor4.2}
Let $\mathcal{I}$ be the set defined in \eqref{eq1.14v34}. Then
\begin{align}\label{eq4.18v63}
\mathcal{I} = \bigcup\limits_{\mu > 0 } \mathcal{I}_\mu,
\end{align}
where
\begin{align*}
\mathcal{I}_\mu : =  \{(\rho,u):Q(\rho)>0, I_\mu ( \rho, u) <l_\mu\},
\end{align*}
and $I_\mu( \rho, u) : =  E(\rho,u)-V_\mu(R_\mu)M(\rho)$.
\end{lemma}

\begin{proof}
From Theorem \ref{thm3.4}, there holds
\begin{align*}
l_\mu = S_\mu (\rho_\mu),
\end{align*}
where $\rho_\mu$ is the spherically symmetric solution of the steady equations \eqref{1.11} 
with center density $\rho_\mu(0)=\mu$.

Let $\rho_\mu(r)=\mu{\rho}_1 \left(\mu^\frac{2-\gamma}{2}r \right)$.
 Then, ${\rho}_1 $ is the spherically symmetric solution of the steady equations \eqref{1.11} with center density ${\rho}_1 (0)=1$.
Hence,  we have
\begin{align}
l_\mu = \mu^\frac{5\gamma-6}{2}\Big[\int_{\mathbb{R}^3}\frac{K\gamma}{\gamma-1}\rho_1^\gamma(x)\, \mathrm{d} x
- \iint_{\mathbb{R}^3\times\mathbb{R}^3}\frac{\rho_1(x)\rho_1(y)}{|x-y|} \, \mathrm{d} x \mathrm{d} y
+\frac{M_1}{R_1}\int_{\mathbb{R}^3}\rho_1(x) \, \mathrm{d} x\Big]
=\mu^\frac{5\gamma-6}{2}l_1.\nonumber
\end{align}
For $( \rho, u) \in \bigcup\limits_{\mu > 0 } \mathcal{I}_\mu$, we have
\begin{align}\label{eq4.19v63}
E(\rho, u) < \max\limits_{\mu >0} f(\mu),
\end{align}
where
\begin{align*}
f(\mu) = V_\mu (R_\mu) M(\rho) + l_\mu
= - \frac{M_1}{R_1} \mu^{\gamma - 1} M(\rho) + \mu^\frac{5 \gamma - 6}2 l_1.
\end{align*}
Direct calculation yields the maximum point $\mu_0$ of $f(\mu)$:
\begin{align*}
\mu_0^\frac{4-3\gamma}2 = \frac{(5\gamma-6)l_1 R_1}{2(\gamma-1)M_1 M(\rho)}.
\end{align*}
We get
\begin{align}\label{eq4.20v63}
f(\mu_0) =   \frac{4-3\gamma}{5\gamma-6} \Big(\frac{5\gamma-6}{2(\gamma-1)}\Big)^\frac{2(\gamma-1)}{4-3\gamma}
 l_1^\frac{2(\gamma-1)}{4-3\gamma}  \Big(\frac{M_1}{R_1}\Big)^\frac{6-5\gamma}{4-3\gamma} M(\rho)^\frac{6-5\gamma}{4-3\gamma}.
\end{align}
Thus, by \eqref{eq4.19v63} and \eqref{eq4.20v63}, we have $(\rho, u) \in \mathcal{I}$ which yields
\begin{align*}
\mathcal{I} \supset  \bigcup\limits_{\mu > 0 } \mathcal{I}_\mu.
\end{align*}
By reversing the argument above, we can show
$\mathcal{I} \subset  \bigcup\limits_{\mu > 0 } \mathcal{I}_\mu.$ Therefore, we obtain \eqref{eq4.18v63}.
\end{proof}
\begin{remark}\label{re4.5}
The set  $\mathcal{I}_\mu$ is non-empty. For the non-rotating star $( \rho_\mu, 0)$ and $ 0 < \lambda< 1$, we have
\begin{align*}
Q \left(\lambda^3\rho_\mu(\lambda x) \right)>0 \ \text{and}\ S_\mu \left(\lambda^3\rho_\mu(\lambda x) \right)<l_\mu.
\end{align*}
Thus
\begin{align*}
 \left(\lambda^3\rho_\mu(\lambda x), 0 \right) \in \mathcal{I}_\mu.
\end{align*}
\end{remark}
\begin{lemma}\label{lem4.1}
The set $\mathcal{I}_\mu$ is invariant under the flow of the compressible Euler-Poisson equations \eqref{1.1}-\eqref{eq1.5v34}.
\end{lemma}
\begin{proof}
Assume the initial data $(\rho_0,u_0)\in \mathcal{I}_\mu $.
We see $I_\mu( \rho(t), u(t)) < l_\mu$ by the conservation of mass and energy. We also have $Q(\rho(t)) > 0$. Indeed, if this is not true, there is $T_0> 0$ such that $Q(\rho(T_0))=0$. But from the law of energy and mass conservation, we have
\begin{align}\label{4.54}
S_\mu(\rho(t))\leq I_\mu( \rho(t), u(t)) = I_\mu ( \rho_0, u_0) < l_\mu,
\end{align}
which contradicts the definition of $l_\mu$. Thus, we have $(\rho(t), u(t)) \in \mathcal{I}_\mu$.
\end{proof}

\begin{lemma}\label{lem4.2}
 For any $\rho$ satisfying $Q(\rho)>0$, we have $Q(\rho)\geq \frac{l_\mu-S_\mu(\rho)}{\lambda^*(\rho)-1}$,
 where $\lambda^*(\rho)$ is defined by \eqref{3.56}.
\end{lemma}
\begin{proof}
For $Q(\rho)>0$, there exists $\lambda^*(\rho)>1$ such that $Q \left(\rho_{\lambda^*(\rho)} \right)=0$ from Lemma \ref{lem3.2} (iii) and (iv).
Therefore, from Lemma \ref{lem3.2} (ii) and \eqref{3.58}, we have
\begin{align*}
l_\mu  \leq S_{\mu,\lambda^*}< S_\mu(\rho)+ \left(\lambda^*-1 \right)\frac{\mathrm{d} S_{\mu,\lambda}}{\mathrm{d}\lambda}\Big|_{\lambda=1}
=S_{\mu}(\rho)+\left(  \lambda^{\ast}-1\right)  Q(\rho).
\end{align*}
It completes the proof.
\end{proof}

\subsection{Existence of spherically symmetric global  weak solutions when $\gamma\in \left(\frac65, \frac43 \right)$}\label{subse4.2v36}
In this subsection, we prove the global existence of a finite energy weak solution when the initial data $(\rho_0,u_0)$ belongs to the invariant set $\mathcal{I}$ by following the ideas in \cite{GLD}.

Firstly, from Lemmas \ref{cor4.2} and \ref{lem4.1}, if the initial data $(\rho_0,u_0)\in\mathcal{I}$, there exists $\mu>0$ such that
$(\rho_0,u_0)\in\mathcal{I}_\mu$.
Following the framework in Subsection \ref{subsubse3.2.1v60},
we will show the existence of global smooth approximate solutions to Navier-Stokes-Poisson equations.

Define
\begin{align}
M(\rho^{\epsilon,b}) = &\int_{a<|x|<b(t)}\rho^{\epsilon,b}  \, \mathrm{d} x, \notag
\intertext{ and }
E \left(\rho^{\epsilon,b},u^{\epsilon,b} \right)=&\frac{1}{2}\int_{a<|x|<b(t)}\rho^{\epsilon,b} \left|u^{\epsilon,b} \right|^2 \, \mathrm{d} x
+\int_{a<|x|<b(t)}\frac{K}{\gamma-1} \left(\rho^{\epsilon,b} \right)^\gamma \, \mathrm{d} x
- \frac1{8 \pi} \int_{\mathbb{R}^3} \left|\nabla V \left(\rho^{\epsilon, b}  \right) \right|^2 \,\mathrm{d}x \label{4.57}  \\
=&\frac{1}{2}\int_{a<|x|<b(t)}\rho^{\epsilon,b}  \left|u^{\epsilon,b} \right|^2 \, \mathrm{d} x
+\int_{a<|x|<b(t)} \frac{K(4-3\gamma)}{\gamma-1}
\left(\rho^{\epsilon,b} \right)^\gamma \, \mathrm{d} x + Q \left(\rho^{\epsilon,b} \right),  \notag
\intertext{ where }
Q \left(\rho^{\epsilon,b} \right)
=& 3K\int_{a<|x|<b(t)} \left(\rho^{\epsilon,b} \right)^\gamma  \, \mathrm{d} x
- \frac1{8 \pi} \int_{\mathbb{R}^3}  \left|\nabla V \left( \rho^{\epsilon, b} \right) \right|^2 \,\mathrm{d}x >0. \notag
\end{align}
In Lemma A.10 of \cite{GLD}, the approximate initial data $ \big(\rho_0^{\epsilon,b},u_0^{\epsilon,b} \big)$ of original initial data $(\rho_0,u_0)\in \mathcal{I}_\mu$ satisfies
\begin{align*}
&M \left(\rho_0^{\epsilon,b} \right)= M(\rho_0),\\
&Q \left(\rho^{\epsilon,b}_0 \right)
=3K\int_{a<|x|<b} \left(\rho_0^{\epsilon,b} \right)^\gamma \, \mathrm{d} x
-  \frac1{8 \pi} \int_{\mathbb{R}^3} \left|\nabla V \left( \rho_0^{\epsilon, b}  \right) \right|^2 \,\mathrm{d}x >0,\\
&E \left(\rho^{\epsilon,b}_0,u^{\epsilon,b}_0 \right)
-V_\mu(R_\mu) M \left(\rho_0^{\epsilon,b} \right)<l_\mu
\end{align*}
for $0<\epsilon\ll 1$ and $b\gg 1$.

Thus, $(\rho_0,u_0)\in\mathcal{I}_\mu$ implies that $ \left(\rho^{\epsilon,b}_0,u^{\epsilon,b}_0 \right)\in\mathcal{I}_\mu$.
We also have
\begin{lemma}
The set $\mathcal{I}_\mu$ is invariant under the flow of \eqref{4.18}-\eqref{4.23} when $\gamma \in \left( \frac65, \frac43\right)$.
\end{lemma}
\begin{proof}
Due to the energy identity of Lemma 3.1 in \cite{GLD}, there holds
\begin{align*}
& \int^{b(t)}_a\bigg(\frac{1}{2}\rho^{\epsilon,b}  \left|u^{\epsilon,b} \right|^2
+\frac{K \left(\rho^{\epsilon,b} \right)^\gamma}{\gamma-1}\bigg)r^2  \, \mathrm{d} r
-\frac{1}{2}\int_a^\infty\frac{1}{r^2}  \bigg( \int^r_a\rho^{\epsilon,b}(z) z^2 \, \mathrm{d} z \bigg)^2 \, \mathrm{d} r\\
&+\epsilon\int^t_0\int^{b(s)}_a \left(\rho^{\epsilon,b}  \left|u^{\epsilon,b}_r \right|^2
+2\rho^{\epsilon,b}\frac{ \left|u^{\epsilon,b} \right|^2}{r^2} \right)r^2 \, \mathrm{d} r \mathrm{d} s
+2\epsilon\int^t_0 \left(\rho^{\epsilon,b}  \left|u^{\epsilon,b} \right|^2 \right)(s,b(s))b(s) \, \mathrm{d} s\\
=&\int^b_a \left(\frac{1}{2}\rho_0^{\epsilon,b}
 \left|u^{\epsilon,b}_0 \right|^2+\frac{K }{\gamma-1} \left(\rho_0^{\epsilon,b} \right)^\gamma \right)r^2 \, \mathrm{d} r
-\frac{1}{2}\int_a^\infty\frac{1}{r^2}\Big(\int^r_a\rho_0^{\epsilon,b}(z)z^2 \, \mathrm{d} z\Big)^2 \, \mathrm{d} r,
\end{align*}
where $ \left(\rho^{\epsilon,b},u^{\epsilon,b} \right)$ is the smooth solution of problem \eqref{4.18}-\eqref{4.23} with initial data
$ \big(\rho_0^{\epsilon,b},u^{\epsilon,b}_0 \big)$. It shows that $E \left(\rho^{\epsilon,b},u^{\epsilon,b} \right)\leq E \big(\rho_0^{\epsilon,b},u^{\epsilon,b}_0 \big)$.
For $0<\epsilon\ll 1,b\gg 1$, by the law of mass conservation, we have
\begin{align*}
M \big(\rho^{\epsilon,b} \big) =4\pi \int_a^b\rho_0^{\epsilon,b}(r)r^2 \, \mathrm{d} r.
\end{align*}
We also have
\begin{align*}
S_\mu \left(\rho^{\epsilon,b} \right)\leq I_\mu \left(\rho_0^{\epsilon,b} \right) <l_\mu,
\end{align*}
where
\begin{align*}
S_\mu \left(\rho^{\epsilon,b} \right)
=\int_{a<|x|<b(t)}\frac{K}{\gamma-1} \left(\rho^{\epsilon,b} \right)^\gamma \, \mathrm{d} x
- \frac1{8\pi} \int_{\mathbb{R}^3} \left|\nabla V \left(\rho^{\epsilon,b} \right) \right|^2 \,\mathrm{d}x
-V_\mu(R_\mu)\int_{a<|x|<b(t)} \rho^{\epsilon,b} \, \mathrm{d} x.
\end{align*}
If there exists $T_0 > 0 $ such that $Q \left(\rho^{\epsilon,b}(T_0) \right)=0$,
it will contradict the definition of $l_\mu$.
\end{proof}

The approximate solutions are constructed by the following results.
\begin{lemma}
For every $\epsilon\in(0,\epsilon_0]$ and $b\gg 1$, assume that the initial data $ \left(\rho^{\epsilon,b}_0,u^{\epsilon,b}_0 \right)\in  \mathcal{I}_\mu$ are smooth and compatible with the boundary conditions \eqref{eq3.20v58} and \eqref{4.23}, $0<C^{-1}_{\epsilon,b}\leq \rho^{\epsilon,b}_0\leq C_{\epsilon,b}<\infty$. Then, there exists a global smooth solution $ \left(\rho^{\epsilon,b},u^{\epsilon,b} \right)$ of \eqref{4.18}-\eqref{4.23}.
\end{lemma}
\begin{proof}
Since $\rho^{\epsilon,b}=0$ for $r\in[0,a) \cup(b(t),\infty)$, from \eqref{4.57}, there holds
\begin{align*}
&\int_{a<|x|<b(t)} \left(\rho^{\epsilon,b} \right)^\gamma \, \mathrm{d} x < \frac{\gamma-1}{4-3\gamma}  E \big(\rho_0^{\epsilon,b},u_0^{\epsilon,b} \big),\\
&  \int_{\mathbb{R}^3}  \left| \nabla V \left( \rho^{\epsilon, b}  \right) \right|^2 \,\mathrm{d}x
\leq C\big(M \big(\rho_0^{\epsilon,b} \big),E \big(\rho_0^{\epsilon,b},u_0^{\epsilon,b} \big)\big).
\end{align*}
The global smooth solutions of the  problem \eqref{4.18}-\eqref{4.23} when $\gamma \in \left( \frac65, \frac43 \right)$ with the  initial data $ \left(\rho^{\epsilon,b}_0,u^{\epsilon,b}_0 \right)$ are established by similar arguments in  \cite{DL} and \cite{KL}. We omit the details.
\end{proof}

From the boundedness of the internal energy and gravitational potential energy of the approximate smooth solutions, by similar arguments as in \cite{GLD}, the global weak solution of the Euler-Poisson equations   in the sense of Definition \ref{def4.7} is obtained by taking $b\rightarrow\infty$ and the vanishing viscosity limit $\epsilon\rightarrow0$.

If $(\rho_0,\frac{\mathcal{M}_0}{\sqrt{\rho_0}})\in\mathcal{I}$, then $(\rho_0,-\frac{\mathcal{M}_0}{\sqrt{\rho_0}})\in\mathcal{I}$. Thus, we have the global  weak solution $ \left(\tilde{\rho}(t,x),\tilde{u}(t,x) \right)=\left(\tilde{\rho},\frac{\widetilde{\mathcal{M}}}{\sqrt{\tilde{\rho}}} \right)$ of \eqref{1.1}-\eqref{eq1.5v34} for $t>0$ with  initial  data $ \left(\rho_0,-\frac{\mathcal{M}_0}{\sqrt{\rho_0}} \right)$. Then, we have that $ (\rho(t,x),u(t,x))= \left(\tilde{\rho}(-t,x),-\tilde{u}(-t,x) \right)$ is the global  weak solution of \eqref{1.1}-\eqref{eq1.5v34} for $t<0$ with initial data $ \left(\rho_0,\frac{\mathcal{M}_0}{\sqrt{\rho_0}} \right)$.

In conclusion, the main result in this part is
\begin{theorem}
When $\gamma\in \left(\frac{6}{5},\frac{4}{3} \right)$,  assume that the initial spherically symmetric data $(\rho_0,\mathcal{M}_0,V_0)$ satisfies \eqref{4.6}, \eqref{4.7} and $ \left(\rho_0,\frac{\mathcal{M}_0}{\sqrt{\rho_0}} \right)\in \mathcal{I}$.
Then, there exists a global  weak solution $(\rho,\mathcal{M},V)$ with spherical symmetry for both $t>0$ and $t<0$.

\end{theorem}

\subsection{Expanding strong solutions}
Assume that the spherically symmetric global strong solution $(\rho,u)$ of  \eqref{1.5}-\eqref{1.8} with the initial data $(\rho_0, u_0) \in \mathcal{I}$ exists. There exists $\mu>0$ such that $(\rho_0,u_0)\in \mathcal{I}_\mu$.
 From \eqref{4.54} and Lemma \ref{lem4.2}, we get
\begin{align}\label{4.77}
Q(\rho(t))&\geq \frac{l_\mu-S_\mu(\rho(t))}{\lambda^*(\rho(t))-1}
\geq \frac{l_\mu-E(\rho_0,u_0)+V_\mu(R_\mu)M(\rho_0)}{\lambda^*(\rho(t))-1}.
\end{align}
If there exists $T_0\in(0, \infty]$ such that $\lim\limits_{t\rightarrow T_0}Q(\rho(t))=0$, we have $\lim\limits_{t\rightarrow T_0}\lambda^*(\rho(t))=1$ from Lemma \ref{lem3.2}. This is a contradiction with \eqref{4.77}.
Hence, there exists a constant $\Lambda>0$ such that
\begin{align*}
Q(\rho(t))\geq \Lambda.
\end{align*}
Then, from
\begin{align*}
E(\rho_0,u_0)=&E(\rho,u)\geq \int_{\mathbb{R}^3}\frac{K\rho^\gamma}{\gamma-1} \, \mathrm{d} x
-\frac{1}{2}\iint_{\mathbb{R}^3\times\mathbb{R}^3}\frac{\rho(x)\rho(y)}{|x-y|}\,  \mathrm{d} x \mathrm{d} y= Q(\rho)+\frac{4-3\gamma}{\gamma-1}\int_{\mathbb{R}^3}K\rho^\gamma\, \mathrm{d} x\\
\geq &\frac{4-3\gamma}{\gamma-1}\int_{\mathbb{R}^3}K\rho^\gamma\, \mathrm{d} x,
\end{align*}
we have  $\int_{\mathbb{R}^3}\rho^\gamma \, \mathrm{d} x<\frac{E(\rho_0,u_0)(\gamma-1)}{K(4-3\gamma)}.$
Note that
\begin{align*}
M(\rho_0)=\int_{\Omega(t)}\rho\, \mathrm{d}x\leq |\Omega(t)|^\frac{\gamma-1}{\gamma}\left(\int_{\Omega(t)}\rho^\gamma\, \mathrm{d}x\right)^\frac{1}{\gamma}.
\end{align*}
There holds
\begin{align*}
|\Omega(t)|\geq \left(\frac{K(4-3\gamma)}{E(\rho_0,u_0)(\gamma-1)}\right)^\frac{1}{\gamma-1}M(\rho_0)^\frac{\gamma}{\gamma-1}>0,
\end{align*}
which shows  the measure of the support is away from $0$.

Let
\begin{align}
H(t)&=\frac{1}{2}\int_{\mathbb{R}^3}\rho|x|^2  \,\mathrm{d} x. \notag
\intertext{ By direct calculation, we have  }
H''(t)& =\int_{\mathbb{R}^3}\rho|u|^2 \, \mathrm{d} x +3K\int_{\mathbb{R}^3}\rho^\gamma \, \mathrm{d} x
-\frac{1}{2}\iint_{\mathbb{R}^3\times\mathbb{R}^3}\frac{\rho(x)\rho(y)}{|x-y|}\,  \mathrm{d} x \mathrm{d} y  \label{eq5.19v51}\\
& =\int_{\mathbb{R}^3} \rho|u|^2 \, \mathrm{d} x + Q(\rho(t)). \notag
\end{align}
By \eqref{eq5.19v51}, we have
\begin{align}\label{4.52}
H''(t)=Q(\rho(t))+ \int_{\mathbb{R}^3} \rho|u|^2 \, \mathrm{d} x\geq \Lambda.
\end{align}
Integrating \eqref{4.52} from $0$ to $t$, there holds
\begin{align*}
 \frac{1}{2}\Lambda t^2+ \Big(\int_{\mathbb{R}^3} \rho_0u_0\cdot x \, \mathrm{d} x\Big)t + \frac{1}{2}\int_{\mathbb{R}^3}\rho_{0}|x|^2  \, \mathrm{d} x
\leq  H(t) = 2\pi\int_{r\leq R(t)}\rho r^4 \, \mathrm{d} r\leq \frac{M(\rho_0)R^2(t)}{2},
\end{align*}
which implies \eqref{1.15} for both $t\in [0, \infty)$ and $t\in(-\infty,0]$.

\begin{remark}
\label{re4.1} We note that the expansion results in Theorems 1.1 and 1.5 still hold true
even when shock waves form. For spherically symmetric
initial data $(\rho_{0},u_{0})$ satisyfing $M(\rho_{0})<M_{K}$ for
$\gamma=\frac{4}{3}$ or $(\rho_{0},u_{0})\in\mathcal{I}$ for $\gamma\in\left(
\frac{6}{5},\frac{4}{3}\right)  ,$ we assume that \eqref{1.5}-\eqref{1.8}
admits a piecewise smooth solution $(\rho,u)\ $containing a single shock
$|x|=s(t)$ with $s(t_{0})=s_{0}$ for some $t_{0}\geq0$. Denote
\[
(\rho,u)\left(  t\right)  =\left\{
\begin{array}
[c]{ll}%
(\rho_{-},u_{-}),\quad if & 0\leq|x|<s(t),\\
(\rho_{+},u_{+}),\quad if & s(t)<|x|<R\left(  t\right)  ,
\end{array}
\right.
\]
where $R\left(  t\right)  $ is the support radius of $\rho\left(  t\right)  $
and $\left(  \rho_{-},u_{-}\right)  \left(  \left\vert x\right\vert ,t\right)
$ and $\left(  \rho_{+},u_{+}\right)  \left(  \left\vert x\right\vert
,t\right)  $ are smooth in the regions $\{(t,x)|\ t\geq0,0\leq|x|<s(t)\}$ and
$\{(t,x)|t\geq0,s(t)<|x|<R\left(  t\right)  \}$, respectively. Since
$(\rho,u)$ is a global weak solution of \eqref{1.5}-\eqref{1.8} by Definition
\ref{def4.7}, it follows that Rankine-Hugoniot conditions at shock location
$|x|=s(t)$ are satisfied, that is,
\begin{align}
&  (P(\rho)+\rho u^{2})(t,s(t)+)-(P(\rho)+\rho u^{2})(t,s(t)-)=(\rho
u(t,s(t)+)-\rho u(t,s(t)-))\dot{s}(t),\label{RH1}\\
&  \rho u(t,s(t)+)-\rho u(t,s(t)-)=(\rho(t,s(t)+)-\rho(t,s(t)-))\dot{s}(t),\\
&  E(\rho,u)|_{(s(t)+,t)}=E(\rho,u)|_{(s(t)-,t)}.\label{RH3}%
\end{align}
By direct calculations and using the Rankine-Hogoniot conditions
\eqref{RH1}-\eqref{RH3}, the following virial identity still holds true
\begin{align*}
H^{\prime\prime}(t)= &  4\pi\int_{0}^{s(t)}\left(  \rho|u|^{2}+3K\rho^{\gamma
}\right)  r^{2}dr+4\pi\int_{s(t)}^{R(t)}\left(  \rho|u|^{2}+3K\rho^{\gamma
}\right)  r^{2}\,dr-\frac{1}{2}\iint_{\mathbb{R}^{3}\times\mathbb{R}^{3}}%
\frac{\rho(x)\rho(y)}{|x-y|}dxdy\\
= &  \int_{\mathbb{R}^{3}}\rho|u|^{2}\,dx+3K\int_{\mathbb{R}^{3}}\rho^{\gamma
}dx-\frac{1}{2}\iint_{\mathbb{R}^{3}\times\mathbb{R}^{3}}\frac{\rho(x)\rho
(y)}{|x-y|}\,dxdy.
\end{align*}
For $\gamma=\frac{4}{3}$, like \eqref{eq4.13v34}, we also have
\[
E(\rho,u)\geq\left(  3K-\frac{1}{2}C_{\min}(M(\rho_{0}))^{\frac{2}{3}}\right)
\int_{\mathbb{R}^{3}}\rho^{\frac{4}{3}}dx>0.
\]
If $\lim_{t\rightarrow+\infty}\int_{\mathbb{R}^{3}}\rho^{\frac{4}{3}}dx=0$,
from
\[
M(\rho_{0})=\int_{\Omega(t)}\rho\,dx\leq|\Omega(t)|^{\frac{1}{4}}\left(
\int_{\Omega(t)}\rho^{\frac{4}{3}}dx\right)  ^{\frac{3}{4}},
\]
it obtains that $\lim_{t\rightarrow+\infty}|\Omega(t)|=+\infty$. If
$\int_{\mathbb{R}^{3}}\rho^{\frac{4}{3}}\,dx\geq C>0$, we have
\[
H^{\prime\prime}(t)=E(\rho,u)+\frac{1}{2}\int_{\mathbb{R}^{3}}\rho
|u|^{2}dx\geq\left(  3K-\frac{1}{2}C_{\min}(M(\rho_{0}))^{\frac{2}{3}}\right)
C>0,
\]
which implies
\[
\frac{1}{2}M(\rho_{0})R^{2}(t)\geq\frac{1}{2}\left(  3K-\frac{1}{2}C_{\min
}(M(\rho_{0}))^{\frac{2}{3}}\right)  Ct^{2}+H^{\prime}(0)t+H(0),
\]
where $R(t)$ is the moving interface of fluids and vacuum states. The
arguments for $\gamma\in\left(  \frac{6}{5},\frac{4}{3}\right)  $ are similar.
Thus, the expansion results still hold in spite of the shock formation.
\end{remark}

\section{ Non-collapsing of white dwarf stars }\label{se5v28}
In this section, we  consider  the white dwarf stars. Recall that the Chandrasekhar limit mass $M_{ch}$ is the supremum of the mass of the non-rotating white dwarf stars. We will first show  $M_{ch}$ is equal to the  mass $M_K$ of the Lane-Emden stars when $\gamma=\frac{4}{3}$ for $K=2AB^{-\frac{4}{3}}$.
Then, we prove the white dwarf stars cannot collapse to a point when the mass is less than the Chandrasekhar limit mass.

For the equation of states $P_w$ of the white dwarf stars defined in \eqref{eq1.16v53}, there exist positive constants $c_1, c_2, \ldots, d_1,d_2, \ldots$ such that
\begin{align*}
P_w(\rho)=\left\{
            \begin{array}{ll}
              2AB^{-\frac{4}{3}} \rho^\frac{4}{3}-c_1\rho^\frac{2}{3}+\ldots, & \text{as}\ \rho\rightarrow \infty, \\
              \\
              d_1\rho^\frac{5}{3}-d_2\rho^\frac{7}{3}+O(\rho^3), & \text{as}\ \rho  \rightarrow 0.
            \end{array}
          \right.
\end{align*}
We can see $P_w(\rho)\approx K \rho^\frac{4}{3}$ for $\rho$ sufficiently large with $K=2AB^{-\frac{4}{3}}$.

As a direct consequence, the enthalpy function of the white dwarf stars
\begin{align}\label{eq5.1v69}
\Phi_w  ( \rho )=\frac{8A}{B}\left[\frac{1}{2}  \left( \frac{\rho}{B} \right)^\frac13  \sqrt{1+ \left( \frac{\rho}{B} \right)^\frac23}+\frac{1}{2}\ln \left(\left( \frac{\rho}{B} \right)^\frac13  +\sqrt{1+ \left( \frac{\rho}{B} \right)^\frac23} \right)\right]
\end{align}
satisfies $\Phi'_w  ( \rho )=\frac{8A}{B^\frac{4}{3}}\left(\sqrt{B^\frac{2}{3}+\rho^\frac{2}{3}}-B^\frac{1}{3}\right)$,
$\Phi''_w  ( \rho )=\frac{8A}{3B^\frac{4}{3}}\frac{\rho^{-\frac{1}{3}}}{\sqrt{B^\frac{2}{3}+\rho^\frac{2}{3}}}$ and
\begin{align}\label{eq5.2v69}
\Phi_w(\rho)=\left\{
            \begin{array}{ll}
        6AB^{-\frac{4}{3}} \rho^\frac{4}{3}+3c_1\rho^\frac{2}{3}+\ldots, & \text{as}\ \rho\rightarrow   \infty, \\
              \\
              \frac{3}{2}d_1\rho^\frac{5}{3}-\frac{3}{4}d_2\rho^\frac{7}{3}+O(\rho^3), & \text{as}\ \rho\rightarrow 0.
            \end{array}
          \right.
\end{align}
Let $F_w(s)= \left(\Phi'_w \right)^{-1}(s)=\left[\left(B^\frac{1}{3}+\frac{B^\frac{4}{3}}{8A}s\right)^2-B^\frac{2}{3}\right]^\frac{3}{2}$ for $s\in(0,s_{\max})$, where
\begin{align*}
s_{\max}=\int^\infty_0\frac{P'_w(\rho)}{\rho} \, \mathrm{d} \rho\in  (0, \infty].
\end{align*}
We extend $F_w(s)$ to $(-\infty,0)$ by zero extension.
The extended function is denoted by $F_w^+(s):\mathbb{R}\rightarrow[0,  \infty)$.
Therefore, the potential $V(\rho)(r)=:V(r)$ satisfies
\begin{align*}
\Delta  V=  V''+ \frac{2}{r}V'=4\pi F_w^+ (V(R)-V(r)).
\end{align*}
Let $y(r)=V(R)-V(r)=\Phi'_w(\rho).$ Then we obtain the following ODE:
\begin{align}\label{6.20}
\begin{cases}
y''+\frac{2}{r}y'=-4\pi F_w^+(y),\\
y(0)=\Phi'_w(\rho(0))=\Phi'_w(\mu)>0,\ y'(0)=0.
\end{cases}
\end{align}

\begin{lemma}\label{le5.1v46}
The Chandrasekhar limit mass $M_{ch}$ equals to the  mass $M_K$ of the Lane-Emden stars  for $\gamma=\frac{4}{3}$ and $K=2AB^{-\frac{4}{3}}$.
\end{lemma}
\begin{proof}
We claim that there exists a constant $\mu_0>0$ such that for any $\mu\in  (\mu_0, \infty)$,
there is a unique solution $y_\mu(r)$ of  \eqref{6.20} satisfying $y_\mu(R_\mu)=0$ for some $R_\mu>0$.
Then, the density of non-rotating  white dwarf stars with support radius $R_\mu$ is
\begin{align*}
\rho_\mu(|x|)=\left\{
                      \begin{array}{ll}
                        F_w(y_\mu(|x|)), & \text{if}~|x|<R_\mu, \\
                        0, & \text{if}~|x|\geq R_\mu.
                      \end{array}
                    \right.
\end{align*}
Indeed, it is equivalent to prove the statement for $\frac{1}{\alpha}=y_\mu(0)=\Phi_w '(\mu)$ sufficiently large.
Define
\begin{align}\label{eq4.3v54}
y_\mu(r)=\frac{1}{\alpha}\theta_\alpha \left(\frac{r}{\alpha} \right),\ s=\frac{r}{\alpha}.
\end{align}
Then, $\theta_\alpha$ satisfies
\begin{align*}
\begin{cases}
\theta''_\alpha+\frac{2}{s}\theta'_\alpha= -g_\alpha (\theta_\alpha),\\
\theta_\alpha(0)=1,\ \theta'_\alpha(0)=0,
\end{cases}
\end{align*}
where
\begin{align}\label{eq4.4v54}
g_\alpha (\theta_\alpha) =  4\pi\alpha^3  F_w^+ \left(\frac{\theta_\alpha}{\alpha} \right)
=
\left\{
\begin{array}{ll}
\left[\left(\alpha B^\frac{1}{3}+\frac{B^\frac{4}{3}}{8A}\theta_\alpha\right)^2-\alpha^2B^\frac{2}{3}\right]^\frac{3}{2},& \theta_\alpha\in[0,1] \\
0, & \theta_\alpha\in(-\infty,0).
\end{array}
\right.
\end{align}
Denote
\begin{align*}
g_0(\theta)=  C_\frac43  \left(\theta^3 \right)_+,
\end{align*}
where $C_\frac43 = \frac{4\pi}{ (4K)^3}$ and $K=2AB^{-\frac{4}{3}}$. One can see
\begin{align}\label{eq4.5v54}
g_\alpha\rightarrow g_0 \text{ in $C^1 [0,1] $ and in $C  (-\infty,1]  $}, \text{ as $\alpha\rightarrow 0^+$}.
\end{align}
Let $\theta_0$ satisfy the following classical Lane-Emden equation:
\begin{align*}
\begin{cases}
\theta''_0+\frac{2}{s}\theta'_0=  -C_\frac43  (\theta_0)_+^3,\\
\theta_0(0)=1,\theta'_0(0)=0.
\end{cases}
\end{align*}
Therefore, for any $R>0$, we have
\begin{align}\label{eq4.6v54}
\theta_\alpha\rightarrow \theta_0 \text{ in $C^1(0,R)$, as $\alpha\rightarrow0$}.
 \end{align}
Define $G(\alpha,s)=\theta_\alpha(s)$ for $\alpha>0$ and $G(0,s)=\theta_0(s)$. Let $R_0$ be the support radius of $\theta_0$.
Note that $G(0,R_0)=\theta_0(R_0)=0$ and $\frac{\partial G}{\partial s}(0, R_0)=\theta'_0(R_0)<0$.
Then, from the implicit function theorem, there exists $\alpha_0>0$ such that for $\alpha\in(0,\alpha_0)$, $G(\alpha,R_\alpha)=0$ where $R_\alpha$ is near $R_0$.
Hence, $R_\alpha$ is the support radius of $\theta_\alpha$.
Thus, for any $\mu>\mu_0 :=F_w \left(\frac{1}{\alpha_0} \right)$,
there exists a unique  solution $y_\mu(r)$ of \eqref{6.20} with the support radius $R_\mu=\alpha R_\alpha$.
This completes the proof of the claim.

Since $F^+_w$ is a $C^1$ function, the solution $  y_\mu(r)$ of \eqref{6.20} is $C^1$ in $r$ and $\mu$.
Therefore, the implicit function theorem implies that $R_\mu$ is $C^1$ in $\mu$.
It also obtains that the density $\rho_\mu(|x|)$ of the unique non-rotating white dwarf star is $C^1$ in $\mu$.
From \cite{Chan1939,ST}, we have for any $\mu\in(0, \infty )$,
there exists a unique non-rotating white dwarf star $(\varrho_\mu,0)$ with center density $\varrho_\mu(0)=\mu\in(0, \infty )$ and the support radius $R_\mu<\infty$.
From \cite{HU,SW}, the mass $M(\mu)$ of the non-rotating white dwarf stars satisfies $\frac{ \mathrm{d} M(\mu)}{ \mathrm{d} \mu}>0$.
Therefore, by the definition of $F_w^+$ and \eqref{eq4.3v54}-\eqref{eq4.6v54}, we obtain
\begin{align*}
M_{ch} =&\sup_{\mu>0}M(\mu)= \lim_{\mu\rightarrow \infty}M(\mu)
=\lim_{\mu\rightarrow \infty}\int_{\mathbb{R}^3}\varrho_\mu  \, \mathrm{d} x
=\lim_{\mu\rightarrow  \infty}4\pi\int_0^\infty F_w^+ (y_\mu) r^2 \, \mathrm{d} r\\
=&\lim_{\alpha\rightarrow0^+}4\pi\int_0^\infty r^2F_w^+ \left(\frac{1}{\alpha}\theta_\alpha \left(\frac{r}{\alpha} \right) \right)\, \mathrm{d} r
=\lim_{\alpha\rightarrow0^+} \int_0^{ \infty}s^2 g_\alpha \left(\theta_\alpha(s) \right) \, \mathrm{d} s
=  4\pi\int_0^{ \infty}s^2\left(\frac{\theta_0(s)}{ 4 K  }\right)^3 \, \mathrm{d} s,
\end{align*}
which is  exactly the  mass $M_K$ of the Lane-Emden stars for $\gamma=\frac{4}{3}$ and $K=2AB^{-\frac{4}{3}}$.
\end{proof}
We now show the white dwarf stars cannot collapse to a point when the mass is less than the Chandrasekhar limit mass.
By \eqref{eq5.1v69}, \eqref{eq5.2v69} and direct calculation, for $\rho'>0$ large enough  and $K=2AB^{-\frac{4}{3}}$, there exist two positive constants $C_1,C_{2,\rho'}$ such that
\begin{align*}
 \left|\int_{\mathbb{R}^3}  \left(\Phi_w( \rho)- 3 K  \rho^\frac43 \right)  \, \mathrm{d} x \right|
 \leq & \int_{ \left\{ x: \rho(x)  \ge  \rho' \right\} }    \left|\Phi_w( \rho)-  3 K   \rho^\frac43 \right|  \, \mathrm{d} x
 + \int_{  \left\{ x: \rho(x)  \le \rho' \right\} }   \left|\Phi_w( \rho)- 3 K  \rho^\frac43 \right|  \, \mathrm{d} x \\
\leq &   C_1 (\rho')^{- \frac23} \int_{\mathbb{R}^3} \rho^\frac43  \, \mathrm{d} x + C_{2,\rho'}\int_{\mathbb{R}^3} \rho \,  \mathrm{d} x.
\end{align*}
Let $(\rho,u)$ be a solution of the Euler-Poisson system with the equation of states $P_w$ and initial data $(\rho_0,u_0)$ of finite  energy and  mass. By \eqref{2.9} and Theorem \ref{thm3.1}, we have
\begin{align}\label{eq4.3v53}
 &   E(\rho_0,u_0)= E(\rho,u)\\
& =   \int_{\mathbb{R}^3} 3 K  \rho^\frac43  \, \mathrm{d} x
+ \int_{\mathbb{R}^3}  \left(\Phi_w(\rho) -  3 K \rho^\frac43 \right) \, \mathrm{d} x
-\frac{1}{2}\iint_{\mathbb{R}^3\times\mathbb{R}^3}\frac{\rho(x)\rho(y)}{|x-y|} \, \mathrm{d} x \mathrm{d} y \notag \\
& \geq  \left( 3 K  - \frac12 C_{\min} \Big(\int_{\mathbb{R}^3} \rho \, \mathrm{d} x\Big)^\frac23
- C_1 ( \rho')^{- \frac23} \right) \int_{\mathbb{R}^3} \rho^\frac43  \, \mathrm{d} x
- C_{2,\rho'} \int_{\mathbb{R}^3} \rho \, \mathrm{d} x. \notag
\end{align}
Note that $C_1$ is a positive constant independent of $\rho'$. Therefore, if $M(\rho_0) < M_{ch}$, we can choose $\rho'$ large enough such that
\begin{align}\label{eq4.4v53}
3 K  - \frac12 C_{\min} (M(\rho_0))^\frac23- C_1 ( \rho')^{- \frac23}>0.
\end{align}
Thus by \eqref{eq4.3v53} and \eqref{eq4.4v53}, we have
\begin{align}\label{eq5.3v41}
\int_{\mathbb{R}^3} \rho^\frac43 \,\mathrm{d} x \le  C (E(\rho_0,u_0), M(\rho_0)).
\end{align}
By the conservation of mass, H\"older's inequality and \eqref{eq5.3v41}, we obtain
\begin{align*}
M(\rho_0)=M(\rho)
\leq|\Omega(t)|^\frac{1}{4}\Big(\int_{\mathbb{R}^3}\rho^\frac{4}{3} \,  \mathrm{d} x\Big)^\frac{3}{4}
\leq \left|\Omega(t) \right|^\frac{1}{4} C(E(\rho_0,u_0), M(\rho_0)),
\end{align*}
where $|\Omega(t)|$ is the measure of the support set of  the white dwarf star. Thus, we have
\begin{align*}
 \left|\Omega(t) \right|\geq \frac{M(\rho_0)}{C(E(\rho_0,u_0), M(\rho_0))}.
\end{align*}
It means that the star cannot collapse to a point. The proof of  Theorem \ref{th5.2} is finished.

\begin{remark}\label{re5.2} The similar conclusion of non-collapse holds for more general equation of states, where $P(s)\in C^1(0,  \infty), P'>0$, $P(0)=0$ and $\lim\limits_{s\rightarrow \infty}s^{-\frac{1}{3}}P'(s)=K>0.$
\end{remark}

\section{High dimensional case of the Euler-Poisson equations}\label{se6v28}
In this section, we  consider the following $n$($n \ge 4$) dimensional extension of \eqref{1.1}-\eqref{1.4}:
\begin{align}\label{eq1.18v36}
&\partial_t \rho+\text{div} \left(\rho \textbf{u} \right)=0,                      \\
&\partial_t  \left(\rho \textbf{u} \right)+\text{div} \left(\rho \textbf{u}\otimes \textbf{u} \right)+ \nabla P (\rho)
=-\rho \nabla V,             \\
&\Delta V= n(n-2)B(n)\rho,      \\
&\lim\limits_{\left\vert x \right\vert \rightarrow\infty}V(t,x)  =0,\label{eq1.18v361}
\end{align}
where $(t,x)\in[0,+\infty)\times\mathbb{R}^n$, $\rho(t,x )\geq 0$  with the support  $\Omega(t)\subset \mathbb{R}^n$ and $\textbf{u}( t, x )\in\mathbb{R}^n$ is the velocity vector. The Poisson equation yields
\begin{align*}
V(x) =-\int_{\mathbb{R}^n}\frac{\rho(y)}{|x-y|^{n-2}}\,  \mathrm{d} y.
\end{align*}
The constant $B(n)$ is the volume of the unit ball in $\mathbb{R}^n$. The pressure is $P(\rho)=K\rho^\gamma (\frac{2n}{n+2}<\gamma\leq \frac{2n-2}{n})$.
The free boundary condition is
\begin{align}
&\rho=0\  \ \text{on}\ \Gamma(t):=\partial\Omega(t),\\
&\mathcal{V}(\Gamma(t))=\mathbf{u}\cdot\mathbf{n}, \label{6.6}
\end{align}
where $\Gamma(t)$ is moving interface of fluids and vacuum states, $\mathbf{n}$ is  exterior unit normal
vector to $\Gamma(t)$ and $\mathcal{V}(\Gamma(t))$ is the  normal velocity of $\Gamma(t)$.

The equations \eqref{eq1.18v36}-\eqref{eq1.18v361} have two conserved quantities, the energy and the mass.
One can see the energy is
\begin{align}
E(\rho,u)
&=\int_{\mathbb{R}^n}\frac{1}{2}\rho | u |^2 \mathrm{d} x+\int_{\mathbb{R}^n}\frac{K\rho^\gamma}{\gamma-1}  \, \mathrm{d} x
-\frac{1}{2}\iint_{\mathbb{R}^n\times\mathbb{R}^n}\frac{\rho(x)\rho(y)}{|x-y|^{n-2}} \, \mathrm{d} x \mathrm{d} y\nonumber\\
&=\int_{\mathbb{R}^n}\frac{1}{2}\rho | u|^2 \, \mathrm{d} x  +\frac{2n-2-n\gamma}{(\gamma-1)(n-2)}\int_{\mathbb{R}^n}\rho^\gamma \,  \mathrm{d} x  +\frac{Q(\rho)}{n-2} ,\label{6.5}
\end{align}
where
\begin{align*}
Q(\rho)=n\int_{\mathbb{R}^n} K\rho^\gamma  \, \mathrm{d}x
-\frac{n-2}{2}\iint_{\mathbb{R}^n\times\mathbb{R}^n}\frac{\rho(x)\rho(y)}{|x-y|^{n-2}} \, \mathrm{d}x\mathrm{d}y.
\end{align*}
The mass is
\begin{align*}
M(\rho)=\int_{\mathbb{R}^n}\rho\,  \mathrm{d}x.
\end{align*}
We can prove similar results as Theorems \ref{th1.1} and \ref{th1.2} for $\gamma = \frac{2n- 2}n$ and $\gamma \in \left( \frac{2n}{n+2}, \frac{2n- 2}n\right)$.
Since the arguments are similar to the proofs of Theorems \ref{th1.1} and \ref{th1.2}, we only give the presentations of some important quantities here. However, different from the results in 3-dimension, we can obtain a  blow-up result for higher dimensional cases.

We can study the similar variational problems for $\gamma = \frac{2(n-1)}n$ and $\gamma \in \left( \frac{2n}{n+2}, \frac{2(n-1)}n \right)$.
Similar to the 3D case, there exists a best constant $C(n,\gamma)>0$ such that
\begin{align*}
\|\nabla V\|^2_{L^2(\mathbb{R}^n)}\leq C(n,\gamma) \|\rho\|_{L^1(\mathbb{R}^n)}^\frac{(n+2)\gamma-2n}{n(\gamma-1)} \|\rho\|^\frac{(n-2)\gamma}{n(\gamma-1)}_{L^\gamma(\mathbb{R}^n)}.
\end{align*}
The virial quantity is
\begin{align*}
H(t)=\frac{1}{2}\int_{\mathbb{R}^n} \rho|x|^2 \, \mathrm{d}x.
\end{align*}
Its first and second derivatives are
\begin{align}
H'(t)&=\int_{\mathbb{R}^n}\rho u \cdot x \, \mathrm{d}x,\nonumber\\
H''(t)&=\int_{\mathbb{R}^n}\rho| u |^2 \, \mathrm{d}x+n\int_{\mathbb{R}^n} K\rho^\gamma \, \mathrm{d}x
-\frac{n-2}{2}\iint_{\mathbb{R}^n\times\mathbb{R}^n}\frac{\rho(x)\rho(y)}{|x-y|^{n-2}} \, \mathrm{d}x\mathrm{d}y\nonumber\\
&=\int_{\mathbb{R}^n}\rho| u |^2\, \mathrm{d}x +Q(\rho)\nonumber\\
&=(n-2)E(\rho, u)+\frac{4-n}{2}\int_{\mathbb{R}^n}\rho|u |^2 \, \mathrm{d}x+\frac{n\gamma-2n+2}{\gamma-1}\int_{\mathbb{R}^n} K\rho^\gamma  \, \mathrm{d}x.\label{H}
\end{align}
Denote
\begin{align*}
&S_\mu(\rho)=\int_{\mathbb{R}^n}\frac{K\rho^\gamma}{\gamma-1} \, \mathrm{d}x
-\frac{1}{2}\iint_{\mathbb{R}^n\times\mathbb{R}^n}\frac{\rho(x)\rho(y)}{|x-y|^{n-2}} \, \mathrm{d}x\mathrm{d}y-V_\mu(R_\mu)\int_{\mathbb{R}^n}\rho \,  \mathrm{d}x,
\end{align*}
where $V_\mu(R_\mu)$ is the value of $V(\rho_\mu)(r)$ at $R_\mu$ and $R_\mu$ is the radius of the support set of the steady solution $\rho_\mu$ to \eqref{eq1.18v36}-\eqref{eq1.18v361} with center density $\rho_\mu(0) = \mu$.
Let
\begin{align*}
\mathcal{I}_\mu = \left\{ ( \rho, u): Q( \rho)> 0, E( \rho, u) - V_\mu (\rho_\mu) M(\rho) < l_\mu \right\}, \\
\mathcal{B}_\mu = \left\{ ( \rho, u): Q( \rho) <  0, E( \rho, u) - V_\mu (\rho_\mu) M(\rho) < l_\mu \right\},
\end{align*}
where $l_\mu= \min \left\{S_\mu ( \rho):\rho \geq 0, \rho \not\equiv 0, Q(\rho)=0 \right\}$.
\begin{theorem}\label{thm6.1}
For $n\geq 4$, consider \eqref{eq1.18v36}-\eqref{eq1.18v361}
with spherically symmetric initial data $(\rho_0,u_0)$ which has compact support and finite energy.
\begin{itemize}
\item (1) If $\gamma= \frac{2n-2}{n}$, assume that the mass of the gaseous star $M(\rho_0)$ satisfies $M(\rho_0)< M_c(n) $ where $M_c(n) $ is the critical mass generated by the best constant $C(n)$ attained by the Lane-Emden stars.
    Then, there exists a global  weak solution with spherical symmetry.
    Moreover, if the spherically symmetric global strong solution of \eqref{eq1.18v36}-\eqref{6.6} exists when $M(\rho_0) < M_c(n)$, we have the support
radius $R(t)>0$  is away from $0$ with
\begin{align*}
R^2(t)\geq  \frac{E(\rho_0, u_0)}{M(\rho_0)}t^2+\frac{2\int_{\mathbb{R}^n}\rho_0 u_0\cdot x \, \mathrm{d}x}{M(\rho_0)}t
+\frac{\int_{\mathbb{R}^n}\rho_0 |x|^2 \, \mathrm{d}x}{M(\rho_0)}.
\end{align*}
(2) If $\gamma\in \left(\frac{2n}{n+2},\frac{2n-2}{n} \right)$,  assume that $(\rho_0, u_0)\in  \mathcal{I}_\mu$ for some $\mu>0$.
 Then, there exists a global  weak solution with spherical symmetry.
Moreover, assume that the spherically symmetric global strong solution of \eqref{eq1.18v36}-\eqref{6.6} exists when $(\rho_0, u_0) \in \mathcal{I}_\mu$.
Then, there exists a constant $\Lambda>0$ such that the support
radius $R(t)>0$  is away from $0$ with
\begin{align*}
R^2(t)\geq  \frac{\Lambda}{M(\rho_0)}t^2
+\frac{2\int_{\mathbb{R}^n}\rho_0 u_0\cdot x \, \mathrm{d}x}{M(\rho_0)}t
+\frac{\int_{\mathbb{R}^n}\rho_0 |x|^2 \, \mathrm{d}x}{M(\rho_0)}.
\end{align*}
\item For $\gamma=\frac{2n-2}{n}$, assume that the mass of gaseous star $M(\rho_0)$ satisfies $M(\rho_0)>M_c(n)$ or for $\gamma\in\left(\frac{2n}{n+2},\frac{2n-2}{n} \right)$, assume that $(\rho_0, u_0)\in   \mathcal{B}_\mu$ for some $\mu>0$.
If $E(\rho_0,u_0)<0$, the solution $(\rho,u)$ of \eqref{eq1.18v36}-\eqref{6.6} blows up in finite time in the sense that there exists $T_0 \in \mathbb{R}_+(T_1\in\mathbb{R}_-) $ such that $\int_{\mathbb{R}^n}  \left|\nabla\sqrt{\rho} \right|^2 \mathrm{d}x\rightarrow \infty$ as $t\rightarrow T_0^-(t\rightarrow T_1^+)$.
\end{itemize}
\end{theorem}

\begin{proof}
We only prove the result of blow-up. If the strong solution $(\rho,u)$ exists and lasts for $t\in[0, \infty)(t\in(-\infty,0])$, from \eqref{H}, we have
\begin{align*}
H''(t)\leq (n-2)E(\rho_0,u_0)<0.
\end{align*}
It is a contradiction with $H(t) \ge 0$ for $t\in[0, \infty)(t\in(-\infty,0])$.
Thus, there is a finite time $0<T_0< \infty(-\infty<T_1<0)$ such that $H(t) \to 0$, as $t \to T_0^-(t \to T_1^+)$.

From the uncertainty principle, we have
\begin{align*}
\frac{3}{2}M(\rho_0)= \frac32 M(\rho)
\leq \left(\int_{\mathbb{R}^n}\rho |x|^2\,  \mathrm{d}x\right)^\frac{1}{2}
\left(\int_{\mathbb{R}^n} \left|\nabla\sqrt{\rho} \right|^2 \, \mathrm{d}x\right)^\frac{1}{2}.
\end{align*}
Hence, $\int_{\mathbb{R}^n} \left|\nabla\sqrt{\rho} \right|^2  \, \mathrm{d}x \to \infty$ as $t \to T_0^-(t \to T_1^+)$.
\end{proof}

\begin{remark}
For $\gamma\in\left(\frac{2n}{n+2},\frac{2n-2}{n} \right)$, similar to the variational arguments in Subsection \ref{subse3.2v31}, we know that $\mathcal{B}_\mu$ is also invariant under the flow of \eqref{eq1.18v36}-\eqref{eq1.18v361}. Similar to Remark \ref{re1.6}, we have $(\lambda^n\rho_\mu(\lambda x),0)\in \mathcal{I}_\mu$ for $0<\lambda<1$ and $(\lambda^n\rho_\mu(\lambda x),0)\in \mathcal{B}_\mu$ for $\lambda>1$ where $(\rho_\mu,0)$ is the non-rotating star. From \eqref{6.5}, $Q(\rho)>0$ implies the energy $E(\rho,u)>0$. One can see that the invariant set $\mathcal{B}_\mu$ contains the data with negative energy. Indeed, let $\rho_\lambda=\lambda^n\rho_\mu(\lambda x)$ with $\lambda>0$. Choosing $\lambda>\left(\int_{\mathbb{R}^n}\frac{K\rho_\mu^\gamma}{\gamma-1}\mathrm{d}x\right)^{\frac{1}{2n-2-n\gamma}}\left(\frac{1}{2}\iint_{\mathbb{R}^n\times\mathbb{R}^n}\frac{\rho_\mu(x)\rho_\mu(y)}{|x-y|^{n-2}}\mathrm{d}x\mathrm{d}y\right)^{-\frac{1}{2n-2-n\gamma}}$, we have
\begin{align*}
E(\rho_\lambda,0)&=\lambda^{n(\gamma-1)}\int_{\mathbb{R}^n}\frac{K\rho_\mu^\gamma}{\gamma-1}\mathrm{d}x-\frac{\lambda^{n-2}}{2}\iint_{\mathbb{R}^n\times\mathbb{R}^n}\frac{\rho_\mu(x)\rho_\mu(y)}{|x-y|^{n-2}}\mathrm{d}x\mathrm{d}y<0,\\
Q(\rho_\lambda)&=\lambda^{n(\gamma-1)}\int_{\mathbb{R}^n}nK\rho_\mu^\gamma\mathrm{d}x-\frac{(n-2)\lambda^{n-2}}{2}\iint_{\mathbb{R}^n\times\mathbb{R}^n}\frac{\rho_\mu(x)\rho_\mu(y)}{|x-y|^{n-2}}\mathrm{d}x\mathrm{d}y\\
&<\lambda^{n(\gamma-1)}\left(n-\frac{n-2}{\gamma-1}\right)\int_{\mathbb{R}^n}K\rho_\mu^\gamma\mathrm{d}x<0.
\end{align*}
Different from the 3D case, the second term in \eqref{H} is non-positive for $n\geq 4$. We can prove that the solution will blow up in finite time when the initial data has negative energy. But it is unknown whether not a solution with arbitrary initial data in $\mathcal{B}_\mu$ will collapse in finite time. Meanwhile, the invariant set $\mathcal{B}_{\mu}$ can be similarly defined in
the 3-D case. It would be very interesting to know if collapsing solutions can be found for initial data in $\mathcal{B}_{\mu}$ for the 3D case.
\end{remark}

\noindent \textbf{Competing interests}
The authors declare that they have no competing interests. All data included in this study are available.

\noindent \textbf{Acknowledgments. }
 Ming Cheng's research was supported in part by the  Scientific and Technological Project of Jilin Provinces Education Department (Grant No. JJKH20231123KJ) and the NSF of China
(Grant No. 11971095 and No. 12371191). Xing Cheng has been partially supported by the NSF of Jiangsu Province (Grant No. BK20221497).
Zhiwu Lin's research was supported partly by the NSF grants DMS-1715201 and DMS-2007457.

\end{document}